\documentclass[11pt,a4paper]{article}

\usepackage{latexsym,amsfonts,amsmath,amsthm,graphics,color}

\makeatletter
\providecommand*{\input@path}{}
\g@addto@macro\input@path{{../common/}}
\makeatother

\usepackage{mathrsfs}
\usepackage{amssymb}

\newcommand{\rest}{\left.\kern-2\nulldelimiterspace\right|_}

\newcommand{\bbE}{{\mathbb E}}

\newcommand{\bbN}{{\mathbb N}}

\newcommand{\bbR}{{\mathbb R}}

\newcommand{\fkS}{{\mathfrak S}}

\newcommand{\rmd}{{\mathrm d}}

\newcommand{\bsb}{{\boldsymbol{b}}}

\newcommand{\bsm}{{\boldsymbol{m}}}

\newcommand{\bsx}{{\boldsymbol{x}}}
\newcommand{\bsy}{{\boldsymbol{y}}}
\newcommand{\bsz}{{\boldsymbol{z}}}

\newcommand{\bszero}{{\boldsymbol{0}}} 
\newcommand{\bsalpha}{{\boldsymbol{\alpha}}}

\newcommand{\bsgamma}{{\boldsymbol{\gamma}}}

\newcommand{\bsnu}{{\boldsymbol{\nu}}}

\newcommand{\bsrho}{{\boldsymbol{\rho}}}

\newcommand{\bssigma}{{\boldsymbol{\sigma}}}

\newcommand{\bsDelta}{{\boldsymbol{\Delta}}}

\newcommand{\calA}{{\mathcal{A}}}

\newcommand{\calC}{{\mathcal{C}}}

\newcommand{\calF}{{\mathcal{F}}}
\newcommand{\calG}{{\mathcal{G}}}

\newcommand{\calJ}{{\mathcal{J}}}

\newcommand{\calL}{{\mathcal{L}}}

\newcommand{\calV}{{\mathcal{V}}}
\newcommand{\calW}{{\mathcal{W}}}
\newcommand{\calX}{{\mathcal{X}}}
\newcommand{\calY}{{\mathcal{Y}}}

\newcommand{\scrK}{{\mathscr{K}}}

\newcommand{\fraku}{{\mathfrak{u}}}

\newcommand{\rd}{\,\mathrm{d}} 

\definecolor{DarkBlue}{rgb}{0,0.08,0.45}
\definecolor{DarkRed}{rgb}{.65,0,0}
\definecolor{applegreen}{rgb}{0.55, 0.71, 0.0}

\newcounter{mymac@matlab}
\newcommand{\matlab}{MATLAB%
   \ifnum\value{mymac@matlab}<1%
   \textregistered%
   \setcounter{mymac@matlab}{1}%
   \fi%
  }

\newtheorem{lemma}{Lemma}
\newtheorem{theorem}{Theorem}

\newtheorem{proposition}{Proposition}

\newtheorem{remark}{Remark}

\theoremstyle{definition}
\newtheorem{definition}{Definition}

\numberwithin{equation}{section}
\numberwithin{definition}{section}
\numberwithin{theorem}{section}
\numberwithin{lemma}{section}
\numberwithin{proposition}{section}
\numberwithin{corollary}{section}

\usepackage[T1]{fontenc}
\usepackage{lmodern}
\usepackage[a4paper]{geometry}
\usepackage{inputenc}
\usepackage{amsmath,amsfonts,amssymb}
\usepackage{bm}
\usepackage{mathtools}
\usepackage{hyperref}
\usepackage{tikz}
\usepackage{caption}
\usepackage{listings}
\usepackage{float}
\usepackage{pgfplots}
\usepackage{subcaption}
\usepackage{enumitem}
\usepackage{xcolor}

\usepackage[normalem]{ulem} 

\usepackage{booktabs}
\usepackage{multirow}

\allowdisplaybreaks

\usepackage{changebar}

\usepackage{cancel}

\begin{document}
\title{Quasi-Monte Carlo integration for feedback control under uncertainty}

\author{Philipp A. Guth\thanks{Johann Radon Institute for Computational and Applied Mathematics,
  \"OAW, Altenbergerstrasse~69, 4040~Linz, Austria.
  (\texttt{philipp.guth@ricam.oeaw.ac.at}, \texttt{peter.kritzer@oeaw.ac.at}).}, \
Peter Kritzer\footnotemark[1], \
Karl Kunisch\footnotemark[1]\ \thanks{Institute of Mathematics and Scientific Computing, Karl-Franzens University Graz,	     	
Heinrichstrasse~36, 8010 Graz. 
(\texttt{karl.kunisch@uni-graz.at})}}

\date{\today}

\maketitle
\begin{abstract}
\noindent A control in feedback form is derived for linear quadratic, time-invariant optimal control problems subject to parabolic partial differential equations with coefficients depending on a countably infinite number of uncertain parameters. It is shown that the Riccati-based feedback operator depends analytically on the parameters provided that the system operator depends analytically on the parameters, as is the case, for instance, in diffusion problems when the diffusion coefficient is parameterized by a Karhunen--Lo\`eve expansion. These novel parametric regularity results allow the application of quasi-Monte Carlo (QMC) methods to efficiently compute an a-priori chosen feedback law based on the expected value. Moreover, under moderate assumptions on the input random field, QMC methods achieve superior error rates compared to ordinary Monte Carlo methods, independently of the stochastic dimension of the problem. Indeed, our paper for the first time studies Banach-space-valued integration by higher-order QMC methods.
\end{abstract}



\section{Introduction}

The design of regulators for controlled dynamical systems is a crucial task across various disciplines. While the governing equations are known in many real-world scenarios, the precise values of certain coefficients within these equations often elude our grasp due to measurement errors, material imperfections, or inherent randomness. A careful analysis of the uncertainty is thus indispensable in order to ensure that the control objective can be achieved.

In the last decade there has been a lot of research on open-loop control problems under uncertainty~\cite{azmi2023analysis,guth2022parabolic,kunoth2013analytic,martinez2016robust}. In these works controls are developed that are optimal with respect to given performance measures, which allow for different levels of risk aversion \cite{KouriShapiro2018,RuszczynskiShapiro}. The performance measures typically involve high-dimensional integrals over the space of uncertain parameters, resulting in computationally challenging problems. Strategies to reduce the computational burden include, for instance, (multilevel) Monte Carlo methods~\cite{Milz2023, NobileVanzan24,VanBarelVandewalle}, (multilevel) quasi-Monte Carlo methods~\cite{guth2022parabolic,Guth2023,LongoSchwabStein}, sparse grids~\cite{BorzivWinckel09,KouriSurowiecCVAR}, and variants of the stochastic gradient descent algorithm~\cite{GeiersbachWollner20,MatthieuKrumscheidNobile21}. We point out that quasi-Monte Carlo methods are particularly well-suited, since they retain the convexity structure of the optimal control problem while achieving faster convergence rates as compared to Monte Carlo methods. 

Control strategies that are directly applicable to different problem configurations, such as varying initial conditions, are highly desirable---especially in the presence of uncertainty. In contrast to open-loop controllers, closed-loop (or feedback) controllers can be constructed independently of the initial condition; hence, this is one reason why they are favourable in the presence of uncertainty. In ~\cite{GuthKunRod23} a robust feedback controller has been developed for stabilization problems with uncertain parameters, and has been applied to tracking problems under uncertainty in~\cite{guthKunRod24}. Adaptive choices of feedback controls for systems with uncertain parameters have been developed in~\cite{GuthKunRod23_2, kramer2017feedback}.

In this manuscript, parameter-dependent tracking problems over a finite time-horizon~$T>0$ for linear autonomous control systems in the form
\begin{align}\label{eq:psys}
\dot{y}_{\bssigma} &= \calA_{\bssigma} y_{\bssigma}+ B u_{\bssigma} + f, \qquad y_{\bssigma}(0) = y_{\circ},
\end{align}
are investigated with state~$y_{\bssigma}(t)\in H$ for every~$\bssigma \in \fkS \subset \bbR^\bbN$, for time~$t\in[0,T]$, and~$\dot y:=\frac{\rmd}{\rmd t}y$. The state space~$H$ is a separable Hilbert space and identified with its own dual. Moreover, the space~$V$ is another separable Hilbert space which is continuously and densely embedded in~$H$, leading to the Gelfand triplet~$V\subset H\subset V'$. The time evolution of the state is described by~$\calA_{\bssigma} \in\calL(V,V')$ depending on a (possibly infinite) sequence of uncertain parameters~$\bssigma \in \fkS$. The control operator~$B\in  \calL(U,H)$ is a bounded linear operator and the control space~$ U$ is a finite-dimensional separable Hilbert space.
The initial condition~$y_\circ \in  H$ and the external forcing~$f \in L^2(0,T;V')$ are given and the choice of the control input~$u \in L^2(0,T; U)$ is  at our disposal.

More precisely, we aim at steering the state~$y_{\bssigma}$ as close as possible to targets~$g \in L^2(0,T;H)$, $g_T \in H$, using a control in feedback form. 
If the exact value of~$\bssigma$ was known, one could follow a classical strategy by considering the minimization of energy-like functionals as
\begin{equation}\label{eq:J1}
\calJ(y_{\bssigma},u_{\bssigma}) = \frac12 \int_0^T \left( \|Q(y_{\bssigma}(t)- g(t))\|^2_{H} +  \|u_{\bssigma}(t)\|^2_{ U} \right) \mathrm dt + \frac{1}{2} \|P(y_{\bssigma}(T)- g_T)\|_{H}^2,
\end{equation}
subject to~\eqref{eq:psys}, for a given observation operator~$Q \in \calL(H)$ and~$P \in \calL(H)$. In this way one obtains an optimal feedback control input~$u_{\bssigma}(t)=K (t,y_{\bssigma}(t))$, with the input feedback operator~$K=K_{\bssigma}$ depending on the parameter~$\bssigma$, arriving at the optimal closed-loop system
\begin{align}
\dot{y}_{\bssigma}(t) &= \calA_{\bssigma} y_{\bssigma}(t)+ BK_{\bssigma} (t,y_{\bssigma}(t)) + f(t), \qquad y_{\bssigma}(0) = y_{\circ}.\notag
\end{align}

If the value of~$\bssigma$ is unknown, one could find an estimate $\overline\bssigma$ for~$\bssigma$. Applying the feedback~$K_{\overline\bssigma}$ designed for the estimate~$\overline\bssigma$, leads to
\begin{align}
\dot{y}_{\bssigma}(t) &= \calA_{\bssigma} y_{\bssigma}(t)+ BK_{\overline\bssigma} (t,y_{\bssigma}(t)) + f(t) \notag\\
&=\calA_{\overline\bssigma} y_{\bssigma}(t)+ BK_{\overline\bssigma} (t,y_{\bssigma}(t))+ f(t)+(\calA_{\bssigma}-\calA_{\overline\bssigma}) y_{\bssigma}(t)\notag.
\end{align}
One may hope that the feedback~$K_{\overline{\bssigma}}$ will have good tracking properties provided that the estimate~$\overline\bssigma$ for~$\bssigma$ is good, in the sense that~$\calA_{\bssigma}-\calA_{\overline\bssigma}$ is small.

However, finding a good estimate can be time consuming, computationally expensive or even impossible in applications. Thus, we propose an input control operator~$K=K_\fkS$, which is based on the expectation with respect to the parameter~$\bssigma \in \fkS$. As a consequence the feedback is independent of a particular realization of the parameter~$\bssigma$ and can be computed a-priori. Moreover, we investigate quasi-Monte Carlo (QMC) approximations of the expected value of the optimal feedback operator:
\begin{align}\label{eq:FBgoal}
    K_\fkS = \int_{\fkS} K(\bssigma) \, \mu(\mathrm d\bssigma) \approx \int_{\fkS_s} K((\bssigma_s,\bszero))\, \mu_s(\mathrm d\bssigma_s) \approx \frac{1}{n} \sum_{k=0}^{n-1} K((\bssigma^{(k)},\bszero)),
\end{align}
where~$\mu$ and~$\mu_s$ are suitable product probability measures over~$\fkS$ and~$\fkS_s \subset \bbR^s$, respectively. The infinite-dimensional integral over~$\fkS$ is first approximated by an~$s$-dimensional integral over~$\fkS_s \subset \bbR^s$, introducing the~\emph{dimension truncation error}. Here, the problem is truncated by setting all components of~$\bssigma$ with index larger than~$s$ to zero, i.e.,~$(\bssigma_s,\bszero):=(\sigma_1,\ldots,\sigma_s,0,0,\ldots)$. The~$s$-dimensional integral is then approximated by an~$n$-point cubature rule with carefully chosen QMC integration points~$\bssigma^{(0)},\ldots,\bssigma^{(n-1)} \in \fkS_s$. QMC rules 
are equal-weight quadrature rules that are frequently used 
for approximating multi-variate integrals on bounded domains. In this context, one tries to make a smart choice of the integration nodes in order to obtain sufficiently fast error convergence for integrands with suitable properties; two prominent classes of integration node sets are lattice point sets and polynomial lattice point sets, which we will both use in this paper. For general introductions to the field, we refer to, e.g., \cite{dick2022lattices, dick2013qmc, dick2010nets, nied1992siam, sloan1994lattices}.

The key contributions of this paper are:
\begin{itemize}
    \item The parametric regularity analysis of the feedback law~$K(\bssigma) = K_{\bssigma}$, consisting of the solution of a differential Riccati equation~$\Pi_{\bssigma}$ associated to the optimal control problem and the solution~$h_{\bssigma}$ of an additional differential equation in the nonhomogeneous case.
    \item Convergence rates for the approximation~\eqref{eq:FBgoal} that are superior to ordinary Monte Carlo rates for sufficiently smooth input randomness, independently of the dimension. As we will show, the integrands considered in this paper are infinitely smooth so that we can obtain error convergence rates of arbitrarily high polynomial order when using suitable integration methods. This manuscript is the first study of higher-order QMC methods for the approximations of integrals with integrands taking values in separable Banach spaces.
    \item The propagation of the total approximation error to control and state trajectories and the quantification of the suboptimality of the proposed feedback~\eqref{eq:FBgoal}.
\end{itemize}

 \subsection{Contents}\label{sS:cont_not}

The manuscript is structured as follows. In Sect.~\ref{sec:paramProb} we state our working assumptions which ensure that the optimality conditions of the parameterized linear quadratic optimal control problem can be formulated as a parameterized family of linear operator equations. In Sect.~\ref{sec:regularity}, under parametric regularity assumptions on~$\calA_\bssigma$, we show that this parameterized family of linear operators falls into the class of~$p$-analytic linear operator equations, which directly provides us with parametric regularity results for the optimal state and adjoint state. In Sect.~\ref{sec:analyticFB} we show that the optimal feedback operator inherits essentially the same parametric regularity as the optimal state and adjoint state, which is then used in the quasi-Monte Carlo error analysis in Sect.~\ref{sec:QMC}. We investigate how the approximation error of the feedback operator propagates into the controls and state trajectories in Sect.~\ref{sec:errorprop}.

\subsection{Notation}
Given real numbers~$r<s$ and separable Banach spaces~$\calX$ and~$\calY$, the space of continuous functions from~$[r,s]$ into~$\calX$ is denoted by~$\calC([r,s];\calX)$ and the Bochner space of strongly measurable square integrable functions from the interval~$(r,s)$ into~$\calX$ is denoted by~$L^2(r,s;\calX)$ and we also denote the subspace~$W(r,s;\calX,\calY):= \{v \in L^2(r,s;\calX)\,\mid\,\dot{v} \in L^2(r,s;\calY)\}$.
Since the time horizon~$T>0$ will be fixed throughout this manuscript, to shorten the exposition, sometimes we shall denote
\begin{equation*}
{\calX}_T := L^2( 0,T;\calX)\quad\mbox{and}\quad W_T(\calX,\calY):= W(0,T; \calX,\calY).
\end{equation*}

By~$\calL(\calX,\calY)$ we denote the space of linear continuous mappings from~$\calX$ into~$\calY$, and in case~$\calX=\calY$ we use the shorter~$\calL(\calX):=\calL(\calX,\calX)$.

Throughout this manuscript, boldfaced symbols are used to denote multi-indices while the subscript notation~$m_j$ is used to refer to the~$j^{\rm th}$ component of a multi-index~$\boldsymbol{m}$. Further, let 
\begin{align*}
    \calF := \{ \bsm \in \bbN_0^\bbN \mid |\bsm| < \infty \}
\end{align*}
denote the set of finitely supported multi-indices, where the order of a multi-index~$\bsm$ is defined as  
$   |\bsm| := \sum_{j\ge 1} m_j$. 
For a sequence~$\bssigma := (\sigma_j)_{j=1}^\infty$ of real numbers and~$\bsm \in \mathcal{F}$, we define
\begin{align*}
        \partial^\bsm_{\bssigma} := \frac{\partial^{m_1}}{\partial\sigma_1} \frac{\partial^{m_2}}{\partial\sigma_2} \cdots, \qquad \text{and} \qquad \bssigma^\bsm := \prod_{j=1}^{\infty} \sigma_j^{m_j},
\end{align*}
where we follow the convention~$0^0 := 1$. Further, we write~$\delta_{\bsm,\boldsymbol{0}} =1$ if~$m_j = 0$ $\forall j \ge1$ and~$\delta_{\bsm,\boldsymbol{0}} = 0$ otherwise.

For positive integers $\ell,m$ with $\ell\le m$ we write $\{\ell:m\}$ to denote the set 
$\{\ell,\ell+1,\ldots,m\}$. Furthermore, for $\fraku\subseteq \{1:s\}$ and a point $\bsx\in [-1,1]^s$, we write $\bsx_{\fraku}$ to denote the projection of $\bsx$ onto those components corresponding to the indices in $\fraku$. For $\bsx=(x_1,\ldots,x_s)$ and $\bsy=(y_1,\ldots,y_s)\in [-1,1]^s$, we write $(\bsx_\fraku \colon \bsy_{\{1:s\}\setminus \fraku})$ to denote the point $\bsz=(z_1,\ldots,z_s)$, where 
\[
 z_j=\begin{cases}
     x_j&\mbox{if $j\in\fraku$,}\\
     y_j&\mbox{otherwise.}
 \end{cases}
\]

In order to indicate that an object~$\mathbb{G}$ depends on a parameter sequence~$\bssigma \in \fkS$, we use both notations~$\mathbb{G}(\bssigma)$ and~$\mathbb{G}_{\bssigma}$ (even for the exact same object) throughout the manuscript.

\section{Optimality conditions of the parameterized linear-quadratic control problem}\label{sec:paramProb}
We assume that the parameter-dependent operator~$\calA_{\bssigma}$ in~\eqref{eq:psys} can be associated with a continuous and~$V$-$H$-coercive parameter-dependent bilinear form~$a(\bssigma;\cdot,\cdot)$, that is
\begin{align}\label{A1}
    \langle \calA_{\bssigma} v,w\rangle_{V',V} := -a(\bssigma;v,w),\qquad \forall v,w \in V,\quad \forall \bssigma \in \fkS,
\end{align}
where~$\exists (\rho,\theta) \in \bbR \times [0,\infty)$ such that
\begin{align}\label{A2}
    a(\bssigma;v,v) + \rho \|v\|_H^2 \ge \theta \|v\|_V^2\quad \forall v \in V\quad {\rm and}\quad \forall \bssigma \in \fkS .
\end{align}
In addition we assume that the operators have a uniform upper bound, that is
\begin{align}\label{A3}
    \|\calA_{\bssigma}\|_{\calL(V,V')}\le C_\calA\quad \forall \bssigma \in \fkS.
\end{align}
Then, we define the parameterized family of parabolic evolution operators~$A_{\bssigma} := A(\bssigma) \in \calL(W_T(V,V'),V'_T \times H)$ as
\begin{align}\label{eq:defpara}
    \langle A_{\bssigma} w, (v_1,v_2) \rangle_{V'_T \times H,V_T \times H} := \langle \dot w, v_1 \rangle_{V_T',V_T} - \langle \calA_{\bssigma} w, v_1 \rangle_{V_T',V_T} + \langle w(0), v_2 \rangle_H
\end{align}
for all~$w\in W_T(V,V')$ and all~$v=(v_1,v_2) \in V_T \times H.$
Using the parabolic operator~$A_{\bssigma}$, the necessary and sufficient optimality conditions for~\eqref{eq:J1} subject to~\eqref{eq:psys} can be written as
\begin{subequations}
    \begin{align}
     A(\bssigma) y(\bssigma) &= \begin{pmatrix}
         B u(\bssigma) + f \\ y_\circ 
    \end{pmatrix} \label{eq:state} \\
    A(\bssigma)^\ast \begin{pmatrix} q_1(\bssigma) \\ q_2(\bssigma) 
    \end{pmatrix}  &= - Q^\ast Q( y(\bssigma) - g) - E_T^\ast P^\ast P (E_Ty(\bssigma) - g_T)\label{eq:adjoint}\\
    u(\bssigma) &=  B^\ast q_1(\bssigma),\label{eq:gradient}
\end{align}
\end{subequations}
where for~$t \in [0,T]$ the operator~$E_t: W_T(V,V') \to H$ evaluates a function at time~$t$. 
Note that~\eqref{eq:state} is an equivalent formulation to~\eqref{eq:psys}. Further, the adjoint equation~\eqref{eq:adjoint}, for every~$\bssigma \in \fkS$, is understood in a weak sense as 
\begin{align*}
     &\langle w, -\dot{q_1}(\bssigma) - \calA_{\bssigma}^\ast q_1(\bssigma) \rangle_{V_T,V_T'} + \langle w(T,\cdot), q_1(\bssigma;T,\cdot) \rangle_H \notag\\
     &\quad- \langle w(0,\cdot), q_1(\bssigma;0,\cdot) \rangle_H + \langle w(0,\cdot), q_2(\bssigma) \rangle_H \\ 
     &= - \langle w, Q^\ast Q (y(\bssigma) - g) \rangle_{H_T} - \langle w(T,\cdot), P^\ast P (y_T(\bssigma) - g_T)\rangle_H 
\end{align*}
for all~$w \in W_T(V,V')$. In particular, testing with~$w \in W_T(V,V')$ satisfying~$w(T,\cdot) = w(0,\cdot) = 0$, leads to~$ -\dot{q_1}(\bssigma) - \calA_{\bssigma}^\ast q_1(\bssigma) = -Q^\ast Q(y(\bssigma) - g)$, and thus~$q_1(\bssigma) \in W_T(V,V')$ for every~$\bssigma\in \fkS$. Then, testing with~$w \in W_T(V,V')$ satisfying~$w(T,\cdot) = 0$, leads to~$q_1(\bssigma;0,\cdot) = q_2(\bssigma)$. Finally, testing with arbitrary~$w \in W_T(V,V')$ leads to~$q_1(\bssigma;T,\cdot) = -P^\ast P(y_T(\bssigma) - g_T)$.

Using the additional regularity of the adjoint variable~$q_1(\bssigma)\in W_T(V,V')$ the optimal state-adjoint-pair~$(y(\bssigma),q_1(\bssigma)) \in W_T(V,V') \times W_T(V,V')$ can be found by solving
\begin{align}\label{eq:Geqn}
    G(\bssigma) 
    \begin{pmatrix}
        y(\bssigma) \\ q_1(\bssigma)
    \end{pmatrix}
    = 
    \begin{bmatrix} f \\y_\circ \\ Q^\ast Q g \\ P^\ast P g_T\end{bmatrix} \quad \in V'_T \times H \times H_T \times H,
\end{align}
where, for every~$\bssigma \in \fkS$, the operator~$G(\bssigma) \in \calL(W_T(V,V')\times W_T(V,V'), V'_T \times H \times V'_T \times H)$ is defined as
\begin{align}\label{eq:saddlepoint}
    G(\bssigma) = \begin{bmatrix} 
    \tfrac{\rm d}{\rm dt} - \calA_{\bssigma} & -BB^\ast \\
    E_0 & 0 \\
    Q^\ast Q & -\tfrac{\rm d}{\rm dt} - \calA^\ast_{\bssigma}\\
    P^\ast P E_T & E_T
    \end{bmatrix}.
\end{align}

In the next section we specify the parametric dependence of the dynamics and then investigate the parametric regularity of several quantities of interest related to the optimal control problem.

\section{Parametric regularity analysis}\label{sec:regularity}

In order to derive parametric regularity results for the feedback law~$K_{\bssigma} = K(\bssigma)$ in Sect.~\ref{sec:analyticFB}, we rely on the concept of~$p$-analytic operators. We shall verify that~$G_{\bssigma}$ is a~$p$-analytic family of operators, utilizing a parametric regularity assumption on~$\calA_{\bssigma}$. To this end, let
\begin{align*}
    \bssigma := (\sigma_1,\sigma_2,\ldots,\sigma_s,\ldots) \in \fkS := \left[-1/2,1/2\right]^\bbN,
\end{align*}
which we equip with the uniform product probability measure
\begin{align*}
    \mu(\mathrm d\bssigma) = \mathrm d\bssigma =  \bigotimes_{j=1}^{\infty} \mathrm d\sigma_j.
\end{align*}

\begin{definition}
Let~$X$ and~$Y$ be two separable Hilbert spaces. A family of bounded linear operators~$\{\mathbb{G}(\bssigma) \in \calL(X,Y') \mid \bssigma \in \fkS\}$ is called~$p$-analytic, for~$p\in (0,1]$, if the following conditions hold:
\begin{itemize}
    \item[(i)] The family of operators~$\mathbb{G}(\bssigma)$ has a uniformly bounded inverse, i.e., there exists a constant~$C_1>0$ such that
    \begin{align}\label{eq:boundedinv}
        \sup_{\bssigma \in \fkS} \|\mathbb{G}(\bssigma)^{-1}\|_{\calL(Y',X)} \leq C_1.
    \end{align}
    \item[(ii)] There exists a nonnegative sequence~$\tilde\bsb = (\tilde b_j)_{j\in \bbN} \in \ell^p(\bbN)$ such that for all~$\bsnu \in \calF\setminus\{\boldsymbol{0}\}$ it holds that
    \begin{align}\label{eq:analytic}
        \sup_{\bssigma \in \fkS} \|\mathbb{G}(\boldsymbol{0})^{-1} (\partial^\bsnu_{\bssigma} \mathbb{G}(\bssigma))\|_{\calL(X,X)} \leq C_1 \tilde\bsb^\bsnu.
    \end{align}
\end{itemize}
\end{definition}

Solutions of linear operator equations are analytic functions of the parameters~$\bssigma$, if the operator is $p$-analytic. This fact is made precise in the following theorem.
\begin{theorem}[{\cite[Thm.~4]{kunoth2013analytic}}]\label{thm:solanalytic}
    Let~$\mathbb{G}(\bssigma)$ be a~$p$-analytic family of operators. Then, for every~$f \in Y'$ and every~$\bssigma \in \fkS$, there is a unique solution~$y(\bssigma)\in X$ of the parameterized operator equation
    \begin{align*}
        \mathbb{G}(\bssigma) y(\bssigma) = f \qquad \text{in} \quad Y'.
    \end{align*}
    Moreover, the parametric solution~$y(\bssigma)$ depends analytically on the parameters, with
    \begin{align*}
        \sup_{\bssigma \in \fkS} \|(\partial^\bsnu_{\bssigma} y)(\bssigma)\|_X \le C_1 \|f\|_{Y'} |\bsnu|! \bsb^\bsnu 
    \end{align*}
    for all~$\bsnu \in \calF$, where~$b_j = \tilde{b}_j/\ln{2}$.
\end{theorem}

\subsection{Affine parameter dependence}\label{sec:affine}
Consider a family of operators~$\mathbb{G}_{\bssigma}$ which depends on the parameters~$\bssigma \in \fkS = [-1/2,1/2]^\bbN$ in an affine manner. This situation arises, e.g., in diffusion problems where the diffusion coefficients are parameterized in terms of a Karhunen--Lo\`{e}ve expansion~\cite{BabsukaNobileTempone,SchwabTodor}. More precisely, consider a family~$(\mathbb{G}_j)_{j\ge 0}$, where~$\mathbb{G}_j \in \calL(X,Y')$ for each~$j\ge0$ such that~$\mathbb{G}_{\bssigma}$ can be represented as
\begin{align*}
    \mathbb{G}(\bssigma) = \mathbb{G}_{0} + \sum_{j \ge 1} \sigma_j \mathbb{G}_j.
\end{align*}
If~$\mathbb{G}_{0} = \mathbb{G}(\boldsymbol{0})$ satisfies~\eqref{eq:boundedinv} with a constant~$\frac{1}{c_{{0}}}$ and the fluctuations~$\mathbb{G}_j$ are small relative to~$\mathbb{G}_{0}$ in the sense that there exists~$0<\kappa<2$ such that~$\sum_{j\ge 1}\|\mathbb{G}_0^{-1}\mathbb{G}_j\|_{\calL(X)} \le \kappa$, then~$\mathbb{G}_{\bssigma}$ satisfies~\eqref{eq:boundedinv} and~\eqref{eq:analytic} with~$C_1 = \frac{1}{(1- \kappa/2)c_{{0}}}$ and~$b_j = \frac{\|\mathbb{G}_j\|_{\calL(X)}}{(1- \kappa/2)c_{{0}}}$, $j\ge 1$, see~\cite[Coro.~1]{Schwab13}.

\subsection{Uniform bounded invertibility of $A$ and $G$}
In this section, we show that the norms of the inverses of the parameterized family of saddle point operators~$\eqref{eq:saddlepoint}$ are uniformly bounded with respect to~$\bssigma \in \fkS$, i.e., satisfy~\eqref{eq:boundedinv}. To do so, we first derive such a bound for the parabolic evolution operators. 

Similar bounds have been obtained in~\cite{kunoth2013analytic}, where it is shown that saddle point operators of linear quadratic control problems subject to~$p$-analytic parameterized parabolic evolution operators are~$p$-analytic, see~\cite[Thm.~22]{kunoth2013analytic}. 

However, the continuity and inf-sup condition constants derived in~\cite{kunoth2013analytic} involve the constant~$\varrho := \sup_{0\ne w\in W_T(V,V')} \frac{\|w(0,\cdot)\|_H}{\|w\|_{W_T(V,V')}}$ depending on the time horizon~$T$. This embedding constant becomes unbounded as~$T \to 0$, rendering it unsuitable for our subsequent analysis. Instead, we derive bounds that are uniform with respect to~$\bssigma \in \fkS$ and remain bounded as~$T \to 0$. Moreover, our bounds depend continuously and monotonically on the time horizon~$T$. This, in conjunction with the time-invariance of the problem, will enable us later in Sect.~\ref{sec:homogeneous} to establish the analytic dependence of the feedback law at any time~$t \in [0, T]$.

\begin{lemma}\label{lem:apriori_unc}
    Under the assumptions~\eqref{A1},~\eqref{A2}, and~\eqref{A3} the parameterized family of parabolic evolution operators~$A_{\bssigma} \in \calL( W_T(V,V') , V'_T \times H)$, as defined in~\eqref{eq:defpara}, has uniformly bounded inverses
    $$
    \|A_{\bssigma}^{-1}\|_{\calL(V'_T \times H,W_T(V,V'))} \le c_A(T) \qquad \forall \bssigma \in \fkS,
    $$ 
    with~$T\mapsto c_A(T)$ continuous and monotonically increasing and independent of~$\bssigma$.
\end{lemma}
The proof of Lemma~\ref{lem:apriori_unc} follows standard arguments and thus is presented in Appendix~\ref{sec:appendix} for completeness.
Under the conditions of Lemma~\ref{lem:apriori_unc}, we next derive a uniform bound for the inverses of the parameterized family of operators~$G_{\bssigma}$ defined in~\eqref{eq:saddlepoint}.

\begin{theorem}\label{thm:contsaddle}
    Under the conditions of Lemma~\ref{lem:apriori_unc}, the parameterized family of operators~$G_{\bssigma} = G(\bssigma) \in \calL(W_T(V,V')\times W_T(V,V'), V_T'\times H \times V'_T \times H)$, as defined in~\eqref{eq:saddlepoint}, has uniformly bounded inverses
    $$
    \|G_{\bssigma}^{-1}\|_{\calL(V_T'\times H \times V'_T \times H,W_T(V,V')\times W_T(V,V'))} \le c_\calG(T),\qquad \forall \bssigma \in \fkS,
    $$
    with
    $$
    c_\calG(T) =  (\mathfrak{C}_y(T)^2 + \mathfrak{C}_q (T)^2)^{\frac{1}{2}},
    $$
independent of~$\bssigma \in \fkS$, and depending continuously and monotonically on~$T$, where
    \begin{align*}
        &\mathfrak{C}_q (T)^2 := c_A(T)^2 2 \left( c_H^2\left(1+ \|Q^\ast Q\|_{\calL(H)} c_V^2 \mathfrak{C}_y(T)^2\right) + 1 + \|P^\ast P\|_{\calL(H)}^2 \left(1+\theta^{-1} + 2\rho T e^{2\rho T}\right)\right),\\
        &\mathfrak{C}_y (T)^2 := c_A(T)^2 (c_H^2\|B\|^2_{\calL(U,H)} \mathfrak{C}_u(T)^2 + 1),\\
        &\mathfrak{C}_u (T)^2 := 2\left(\|Q\|_{\calL(H)}^2 c_V^2 c_A(T)^2 + \|P\|_{\calL(H)}^2\left( 1+\theta^{-1} + 2\rho Te^{2\rho T}\right) + 4\phi(T)\right)
    \end{align*}
    with~$\mathfrak{C}_1 := 1+\theta^{-1} + 2\rho T e^{2\rho T}$, and~$c_H$ and~$c_V$ denoting the embedding constant~$\|\cdot\|_{V'} \le c_H \|\cdot\|_H$ and ~$\|\cdot\|_{H} \le c_V \|\cdot\|_{V}$, respectively.
\end{theorem}
\begin{proof}
For every~$\bssigma \in \fkS$, we will show that, for arbitrary functions $(a,b,c,d) \in V'_T \times H \times V'_T \times H$, there exists a unique solution~$(w_1,w_2) \in W_T(V,V') \times W_T(V,V')$ of
\begin{align}\label{eq:generalop}
    G(\bssigma) \begin{pmatrix}
    w_1(\bssigma)\\w_2(\bssigma)
\end{pmatrix} = \begin{bmatrix}
    a\\b\\c\\d
\end{bmatrix},
\end{align}
which depends continuously on~$(a,b,c,d)$. First, observe that~\eqref{eq:generalop} can be written as
\begin{subequations}\label{eq:genOCprob}
\begin{align}
\dot w_1(\bssigma) - \calA(\bssigma) w_1(\bssigma) &= a + B u(\bssigma) \label{eq:genOCprob1}\\
w_1(\bssigma;0) &= b\label{eq:genOCprob2}\\
-\dot w_2(\bssigma) - \calA(\bssigma) w_2(\bssigma)  &= -(Q^\ast Q w_1(\bssigma) - c)\label{eq:genOCprob3}\\
w_2(\bssigma;T) &= -(P^\ast P w_1(\bssigma;T)-d)\label{eq:genOCprob4}\\
u(\bssigma) &= B^\ast w_2(\bssigma).
\end{align}
\end{subequations}
Moreover, for every~$\bssigma \in \fkS$, the set of equations~\eqref{eq:genOCprob} defines necessary and sufficient optimality conditions of the linear quadratic optimal control problem~$\min_{w_1(\bssigma),u(\bssigma)} J(w_1(\bssigma),u(\bssigma))$, subject to~\eqref{eq:genOCprob1} and~\eqref{eq:genOCprob2}, where
\begin{align}
J(w_1(\bssigma),u(\bssigma)) &:= \frac12\left( \|Qw_1(\bssigma)\|^2_{H_T} +  \|u(\bssigma)\|^2_{ U_T} + \|P w_1(\bssigma;T)\|_{H}^2 \right)\notag \\
&\quad - \langle w_1(\bssigma), c\rangle_{V_T,V_T'}  - \langle w_1(\bssigma;T),d\rangle_H \notag.
\end{align}
The existence of a minimizer~$(w_1(\bssigma),u(\bssigma))$ of this linear quadratic optimal control problem (and thus the existence of a solution~$(w_1(\bssigma),w_2(\bssigma))$ of~\eqref{eq:generalop}) follows by the direct method in calculus of variations. Furthermore, the minimizer is unique since~$J$ is strictly convex.

Next, we will prove that~$G(\bssigma)^{-1}$ is uniformly bounded with respect to~$\bssigma \in \fkS$. For this purpose, it remains to obtain a uniform bound on~$w_1(\bssigma),w_2(\bssigma)$ in terms of~$(a,b,c,d)$.

To this end, denoting by~$\bar w_1(\bssigma)$ the solution of the uncontrolled system~\eqref{eq:eqn}, we observe that, for every~$\bssigma \in \fkS$, the optimal state-control pair~$(w_1(\bssigma),u(\bssigma))$ satisfies~$J(w_1(\bssigma),u(\bssigma)) \le J(\bar w_1(\bssigma),0)$. Thus,
\begin{align}
    \frac{1}{2} \|u(\bssigma)\|^2_{U_T} &\le \frac12 \left(\|Q\bar w_1(\bssigma)\|^2_{H_T} - \|Q w_1(\bssigma)\|^2_{H_T} + \|P \bar w_1(\bssigma;T)\|_{H}^2 -  \|P w_1(\bssigma;T)\|_{H}^2\right)\notag\\
    &\quad - \langle \bar w_1(\bssigma) - w_1(\bssigma), c\rangle_{H_T} - \langle \bar w_1(\bssigma;T) - w_1(\bssigma;T), d\rangle_{H}\notag \\
    &\le \frac12 \left(\|Q\bar w_1(\bssigma)\|^2_{H_T} + \|P \bar w_1(\bssigma;T)\|_{H}^2 \right)\notag\\
    &\quad+ \frac{\varepsilon}{2} \| \bar w_1(\bssigma) - w_1(\bssigma)\|^2_{H_T} + \frac{1}{2\varepsilon} \| c\|^2_{H_T} + \frac{\varepsilon}{2} \| \bar w_1(\bssigma;T) - w_1(\bssigma;T)\|^2_{H} + \frac{1}{2\varepsilon} \| d\|^2_{H},\notag
\end{align}
for an arbitrary~$\varepsilon>0$ with Young's inequality. Denoting by~$c_V$ the embedding constant~$\|\cdot\|_{H} \le c_V \|\cdot\|_V$, we obtain
\begin{align}
    \|\bar w_1(\bssigma) - w_1(\bssigma) \|_{H_T}^2 \le c_V^2 \|\bar w_1(\bssigma) - w_1(\bssigma)\|_{V_T}^2\notag \le c_V^2 \|\bar w_1(\bssigma) - w_1(\bssigma)\|_{W_T(V,V')}^2.
\end{align}
We note that by the linearity of the state equation,~$\bar w_1(\bssigma) - w_1(\bssigma)$ solves~\eqref{eq:eqn} with~$f = Bu(\bssigma)$ and~$y_\circ = 0$. Thus, from Lemma~\ref{lem:apriori_unc} it follows that
\begin{align}
    \|\bar w_1(\bssigma) - w_1(\bssigma) \|_{H_T}^2 \le c_V^2 c_A(T) \|Bu(\bssigma)\|_{V'_T}^2 \le c_V^2 c_A(T)  c_H^2 \|B\|^2_{\calL(U,H)} \|u(\bssigma)\|^2_{U_T},
\end{align}
where~$c_H$ denotes the embedding constant~$\|\cdot\|_{V'} \le c_H \|\cdot\|_H$. Further, from~\eqref{eq:ex1help2} and~\eqref{eq:yT-bound} below we conclude that
\begin{align}
    \| \bar w_1(\bssigma;T) - w_1(\bssigma;T)\|^2_{H} 
    \le \left(\frac{1}{\theta} + 2\rho T e^{2\rho T} \right) c_H^2 \|B\|^2_{\mathcal{L}(U,H)} \|u(\bssigma)\|^2_{U_T}.\notag
\end{align}
Setting~$\phi(T) := \left(\frac{1}{\theta} +  2\rho T e^{2\rho T} + c_V^2 c_A(T)\right) c_H^2 \|B\|_{\calL(U,H)}^2$, we arrive at
\begin{align}
    \left(\frac{1}{2} - \varepsilon \phi(T)\right) \|u(\bssigma)\|_{U_T}^2 \le \frac12 \left(\|Q\bar w_1(\bssigma)\|^2_{H_T} + \|P \bar w_1(\bssigma;T)\|_{H}^2 +\frac{1}{\varepsilon} \|c\|_{H_T}^2 + \frac{1}{\varepsilon} \|d\|_H^2\right),\notag
\end{align}
and selecting~$\varepsilon = \frac{1}{4 \phi(T)}$ gives
\begin{align}
    \|u(\bssigma)\|_{U_T}^2 &\le 2 \left(\|Q\bar w_1(\bssigma)\|^2_{H_T} + \|P \bar w_1(\bssigma;T)\|_{H}^2 + 4\phi(T) \left(\|c\|_{H_T}^2 +\|d\|_H^2\right)\right) \notag\\
    &\le 2 \left(\|Q\|_{\calL(H)}^2 \|\bar w_1(\bssigma)\|^2_{H_T} + \|P\|^2_{\calL(H)} \|\bar w_1(\bssigma;T)\|_{H}^2 +  4\phi(T) \left(\|c\|_{H_T}^2 +\|d\|_H^2\right)\right).\label{eq:ubound}
\end{align}
Moreover, using \eqref{eq:yT-bound} together with~\eqref{eq:ex1help2} gives
\begin{align}
    \|\bar w_1(\bssigma;T)\|_{H}^2 &\le \|b\|^2_H + 2\rho T e^{2\rho T}(\|b\|_H^2 + \|a\|_{V'_T}^2) + \frac{1}{\theta} \|a\|_{V'_T}^2\notag \\
    &\le (1+\frac{1}{\theta} + 2\rho T e^{2\rho T}) (\|b\|_H^2 + \|a\|_{V'_T}^2) \label{eq:yTest}.
\end{align}
Further, with~$\|\bar w_1(\bssigma)\|_{H_T}^2 \le c_V^2 \|\bar w_1(\bssigma)\|_{V_T}^2 \le c_V^2 \|\bar w_1(\bssigma)\|_{W_T(V,V')}^2$ and Lemma~\ref{lem:apriori_unc}, we have that~$\|\bar w_1(\bssigma)\|_{H_T}^2 \le c_V^2 c_{A}(T)^2 (\|b\|_H^2 + \|a\|_{V'_T}^2)$. Then, together with~\eqref{eq:ubound} and~\eqref{eq:yTest} we conclude that
\begin{align*}
    \|u(\bssigma)\|^2_{U_T} 
    &\le 2\left(\|Q\|_{\calL(H)}^2 c_V^2 c_A(T)^2 + \|P\|_{\calL(H)}^2( 1+\frac{1}{\theta} + 2\rho Te^{2\rho T})\right)(\|b\|_H^2 + \|a\|_{V'_T}^2)\\&\quad+ 8\phi(T) \left(\|c\|_{H_T}^2 +\|d\|_H^2\right),
\end{align*}
and thus
\begin{align}\label{eq:enbound}
    \|u(\bssigma)\|_{U_T}^2 \leq \mathfrak{C}_u(T)^2 (\|a\|_{V'_T}^2 + \|b\|_H^2 + \|c\|_{H_T}^2 + \|d\|_{H}^2),
\end{align}
with~$\mathfrak{C}_u(T)$ as above.

From~\eqref{eq:enbound} and Lemma~\ref{lem:apriori_unc} it follows that
\begin{align}
\|w_1(\bssigma)\|_{W_T(V,V')}^2 &\le c_A(T)^2 (\|Bu(\bssigma) + a\|_{V'_T}^2 + \|b\|_H^2)\notag \\ &\le \mathfrak{C}_y(T)^2(\|a\|_{V'_T} + \|b\|_H^2 + \|c\|_{V'_T}^2 + \|d\|_{H}^2).\notag
\end{align}
for every~$\bssigma \in \fkS$ with~$\mathfrak{C}_y$. 
Analogously to~\eqref{eq:uncapriori} we obtain for the adjoint equation~\eqref{eq:genOCprob3}--\eqref{eq:genOCprob4} the following estimate
\begin{align}
\|w_2(\bssigma)\|_{W_T(V,V')}^2 &\le c_A(T)^2 \left(\|Q^*Q w_1(\bssigma)-c\|_{V'_T}^2 + \|P^* P w_1(\bssigma;T)-d\|_H^2\right)\notag\\
&= \mathfrak{C}_q^2 (\|a\|^2_{V'_T} + \|b\|_H^2 + \|c\|_{V'_T}^2+ \|d\|_{H}^2), \notag
\end{align}
for every~$\bssigma \in \fkS$. We conclude that for every~$\bssigma \in \fkS$ it holds that
\begin{align}
    \|(w_1(\bssigma),w_2(\bssigma))\|_{W_T(V,V')\times W_T(V,V')}^2 &= \|w_1(\bssigma)\|^2_{W_T(V,V')} + \|w_2(\bssigma)\|^2_{W_T(V,V')} \notag\\
    &\le (\mathfrak{C}_y^2 + \mathfrak{C}_q^2) (\|a\|^2_{V'_T} + \|b\|_H^2 + \|c\|_{V'_T}^2+ \|d\|_{H}^2), \notag
\end{align}
which shows the desired result.
\end{proof}

\subsection{Parametric regularity of $A$ and $G$}
In Theorem~\ref{thm:contsaddle} we have shown that~$G_{\bssigma}^{-1}$ is uniformly bounded with respect to~$\bssigma \in \fkS$ by a constant~$c_\calG(T)$ provided that~$A_{\bssigma}^{-1}$ is uniformly bounded by a constant~$c_A(T)$. Similarly, the following result shows that the parametric regularity of~$\calA_{\bssigma}$ determines the parametric regularity of~$G_{\bssigma}$.

\begin{theorem}\label{thm:s_saddle_op}
    Let the assumptions~\eqref{A1},~\eqref{A2}, and~\eqref{A3} hold, and assume that there exists a sequence~$\tilde \bsb \in \ell^p(\bbN)$, for some~$0<p\le1$, of nonnegative numbers, such that~$\|\partial_{\bssigma}^\bsnu \calA_{\bssigma}\|_{\calL(V,V')} \le \tilde \bsb^{\bsnu}$ for all~$\bsnu \in \calF \setminus \{\boldsymbol{0}\}$.
    Then, for every~$\bssigma \in \fkS$, the tracking problem of minimizing~$\calJ(y_{\bssigma},u_{\bssigma})$ subject to~\eqref{eq:psys} over~$(y_{\bssigma},u_{\bssigma})$ can be formulated as an operator equation~\eqref{eq:Geqn}, and the associated operator, as defined in~\eqref{eq:saddlepoint}, is~$p$-analytic with the same regularity parameter~$p$.
    Moreover, from Theorem~\ref{thm:solanalytic} it follows that the state and adjoint state depend analytically on the parameters~$\bssigma \in \fkS$:
    \begin{align}\label{eq:state-adjoint-analytic}
        \left\| \partial^\bsnu_{\bssigma} \begin{pmatrix} y\\q_1 \end{pmatrix} (\bssigma)\right\|_{W_T(V,V')\times W_T(V,V')} \le c_\calG(T) |\bsnu|! \bsb^\bsnu \left\|\begin{pmatrix}f \\ y_\circ \\ Q^\ast Qg \\P^\ast P g_T\end{pmatrix}\right\|_{V'_T \times H \times V'_T \times H},
    \end{align}
    for all~$\bsnu \in \calF$, with~$b_j := \tilde{b}_j/\ln{2}$, and a constant~$c_\calG(T)>0$ depending continuously and monotonically increasing on~$T$ and which is independent of~$\bssigma \in \fkS$. 
\end{theorem}
\begin{proof}
    From Lemma~\ref{thm:contsaddle} we know that~$G_{\bssigma}$ satisfies~\eqref{eq:boundedinv} with constant~$c_\calG(T)>0$. Hence, it remains to show that there is a nonnegative sequence~$\tilde{\bsb} \in \ell^p(\bbN)$, such that for all~$\bsnu \in \calF \setminus \{\boldsymbol{0}\}$, the operator~$G_{\bssigma}$ satisfies~\eqref{eq:analytic}.

    In order to prove this, we observe that for all~$\bsnu \in \calF \setminus \{\boldsymbol{0}\}$ and for all~$(w_1,w_2) \in W_T(V,V') \times W_T(V,V')$ we have
\begin{align}
    \|G(\boldsymbol{0})^{-1} \partial_{\bssigma}^\bsnu G(\bssigma) (w_1,w_2)\|^2_{W_T(V,V')\times W_T(V,V')} &\le c_\calG(T)^2 \|\partial_{\bssigma}^\bsnu G(\bssigma) (w_1,w_2)\|_{V'_T \times H \times V'_T \times H}^2 \notag\\
    &\!\!\!\!\!\!\!\!\!\!\!\!\!\!\!\!\!\!\!\!\!\!\!\!\!\!\!\!\!\!\!\!\!\!\!\!\!\!\!\!\!\!\!\!\!\!\!\!\!\!\!\!\!\!\!= c_\calG(T)^2 \left( \|\partial_{\bssigma}^\bsnu \calA_{\bssigma} w_1\|_{V'_T}^2 + \|\partial_{\bssigma}^\bsnu \calA^\ast_{\bssigma}  w_2\|_{V'_T}^2\right) \notag \\
    &\!\!\!\!\!\!\!\!\!\!\!\!\!\!\!\!\!\!\!\!\!\!\!\!\!\!\!\!\!\!\!\!\!\!\!\!\!\!\!\!\!\!\!\!\!\!\!\!\!\!\!\!\!\!\!\le c_\calG(T)^2  \|\partial_{\bssigma}^\bsnu \calA_{\bssigma}\|^2_{\calL(V,V')} \|(w_1,w_2)\|_{W_T(V,V') \times W_T(V,V')}^2.\notag
\end{align}
Thus, it holds that
\begin{align}
    \|G(\boldsymbol{0})^{-1} \partial_{\bssigma}^\bsnu G(\bssigma) \|_{\calL(W_T(V,V')\times W_T(V,V'))} \le  c_\calG(T)  \tilde{\bsb}^\bsnu,\notag
\end{align}
and~$G_{\bssigma}$ is~$p$-analytic with the constant~$ c_\calG(T)$.
\end{proof}

We remark that the condition that there exists a sequence~$\tilde \bsb \in \ell^p(\bbN)$, for some~$0<p\le1$, of nonnegative numbers, such that~$\|\partial_{\bssigma}^\bsnu \calA_{\bssigma}\|_{\calL(V,V')} \le \tilde \bsb^{\bsnu}$ for all~$\bsnu \in \calF \setminus \{\boldsymbol{0}\}$ follows from assumption~\eqref{A3} provided that~$\calA_{\bssigma}$ satisfies~\eqref{eq:analytic}. In particular, if~$\calA_\bssigma$ satisfies~\eqref{eq:analytic} for some sequence~$\bsrho$, then we can take~$\tilde b_j = C_{\calA}C_1\rho_j$, $j\ge 1$.

Since~$\|\cdot\|^2_{W_T(V,V')\times W_T(V,V')} = \|\cdot\|_{W_T(V,V')}^2 + \|\cdot\|_{W_T(V,V')}^2$, the bound~\eqref{eq:state-adjoint-analytic} holds in particular for the state and adjoint individually, that is
\begin{align}
    \|(\partial^\bsnu_{\bssigma} y)(\bssigma)\|_{W_T(V,V')} \le c_\calG(T) |\bsnu|!   \bsb^\bsnu \|(f , y_\circ, Q^\ast Qg, P^\ast Pg_T)\|_{V'_T \times H \times V'_T \times H}\label{eq:analyticstate}
    \intertext{and}
    \|(\partial^\bsnu_{\bssigma} q_1)(\bssigma)\|_{W_T(V,V')} \le c_\calG(T) |\bsnu|!   \bsb^\bsnu \|(f , y_\circ, Q^\ast Qg, P^\ast Pg_T)\|_{V'_T \times H \times V'_T \times H}\label{eq:analyticadjoint}.
\end{align}

\subsection{Parametric regularity of the optimal cost}
In order to establish the parametric regularity of the feedback law in the next section, we first investigate the parametric regularity of the following quantities of interest (QoI) related to the objective functional~\eqref{eq:J1}:
\begin{align}\label{eq:QoIs}
\big|\partial^\bsnu_{\bssigma} \|Q(y(\bssigma)-g)\|_{H_T}^2\big| \quad {\rm and}\quad  \big|\partial^\bsnu_{\bssigma} \|u(\bssigma)\|_{U_T}^2\big| \quad {\rm and}\quad \big|\partial^\bsnu_{\bssigma} \|P(y_T(\bssigma) - g_T)\|_H^2\big|.
\end{align}

For future reference, we include the general cases~$0 \ne g\in H_T$ and~$0 \ne g_T \in H$ in the following regularity analysis: for~$g$ independent of~$\bssigma$, we have~$\partial^\bsnu_{\bssigma} g = g$ if~$\bsnu=\boldsymbol{0}$ and~$\partial^\bsnu_{\bssigma} g = 0$ otherwise. 

\begin{lemma}\label{lem:objregular}
    Let the assumptions of Theorem~\ref{thm:s_saddle_op} hold. Then, it holds that
    \begin{subequations}
        \begin{align*}
        \big|\partial^\bsnu_{\bssigma} \|Q(y(\bssigma)-g)\|_{H_T}^2\big| &\le C_1(T) (|\bsnu|+1)!\, \bsb^\bsnu \\
        \big|\partial^\bsnu_{\bssigma} \|u(\bssigma)\|_{U_T}^2\big| &\le C_2(T) (|\bsnu|+1)! \,\bsb^\bsnu \\
        \big|\partial^\bsnu_{\bssigma} \|P(y_T(\bssigma) - g_T(\bssigma))\|_H^2\big| & \le C_3(T) (|\bsnu|+1)! \,\bsb^\bsnu , 
    \end{align*}
    for all~$\bsnu \in \calF$, where
    \begin{align}
        C_1(T) &:= \|Q\|_{\calL(H)}^2\big(c_V c_\calG(T) \|(f , y_\circ, Q^\ast Qg, P^\ast P g_T)\|_{V'_T \times H \times V'_T \times H}+\|g\|_{H_T}\big)^2\notag \\
        C_2(T) &:= \|B\|_{\calL(U,H)}^2 c_V^2 c_\calG(T)^2\|(f , y_\circ, Q^\ast Qg, P^\ast P g_T)\|_{V'_T \times H \times V'_T \times H}^2 \notag\\
        C_3(T) &:= \|P\|_{\calL(H)}^2 (\|g_T\|_H + (c_\calG(T)^2 +1)\|(f,y_\circ,Q^\ast Q g, P^\ast P g_T)\|_{V_T' \times H \times V_T' \times H})^2. \notag
    \end{align}
    \end{subequations}
    In particular, the following regularity result holds for the optimal cost,
    \begin{align}
        \big|\partial^\bsnu_{\bssigma} \calJ(y_{\bssigma},u_{\bssigma})\big| \le \frac{C_4(T)}{2} (|\bsnu|+1)! \bsb^\bsnu,
    \end{align}
    for all~$\bsnu \in \calF$ with~$C_4(T) =\sum_{i=1}^3 C_i(T)$ depending continuously and monotonically increasing on~$T$.
\end{lemma}
\begin{proof}
Let~$\bsnu \in \calF$. We begin with the first QoI in~\eqref{eq:QoIs}. Using that~$Q$ is independent of~$\bssigma \in \fkS$, with~\eqref{eq:analyticstate}, we have
\begin{align}
    \|\partial^\bsnu_{\bssigma} Q(y(\bssigma)-g)\|_{H_T} &\le \big(\|\partial^\bsnu_{\bssigma}  Qy(\bssigma)\|_{H_T} + \|\partial^\bsnu_{\bssigma} Qg\|_{H_T} \big) \le C_g |\bsnu|! \bsb^\bsnu ,\label{eq:helpterm1}
\end{align}
where~$C_g := \|Q\|_{\calL(H)}\big(c_V c_\calG(T) \|(f , y_\circ, Q^\ast Qg, P^\ast P g_T)\|_{V'_T \times H \times V'_T \times H}+\|g\|_{H_T}\big)$. Then, for the first QoI in~\eqref{eq:QoIs} we have
\begin{align}
    \big|\partial^\bsnu_{\bssigma} \|Q(y(\bssigma) - g)\|_{H_T}^2\big| &= \big|\partial^\bsnu_{\bssigma} \langle Q(y(\bssigma) - g) , Q(y(\bssigma) - g) \rangle_{H_T}\big| \notag \\
    &=\big| \sum_{\bsm \le \bsnu} \binom{\bsnu}{\bsm} \langle \partial^\bsm_{\bssigma} Q(y(\bssigma)-g), \partial^{\bsnu-\bsm}_{\bssigma} Q(y(\bssigma)-g)\rangle_{H_T}\big|. \notag
\end{align}
Using the Cauchy--Schwarz inequality, the triangle inequality, and~\eqref{eq:helpterm1} we obtain
\begin{align}
    \big|\partial^\bsnu_{\bssigma} \|Q(y(\bssigma) - g)\|_{H_T}^2\big| &\le \sum_{\bsm \le \bsnu} \binom{\bsnu}{\bsm} C_g |\bsm|! \bsb^\bsm  C_g |\bsnu -\bsm|! \bsb^{\bsnu-\bsm}   \notag\\
    &= C_g^2 \bsb^\bsnu \sum_{\ell = 0}^{|\bsnu|} |\ell|! (|\bsnu|-\ell)! \sum_{\substack{|\bsm|=\ell\\ \bsm \le \bsnu}} \binom{\bsnu}{\bsm} = C_g^2 \bsb^\bsnu (|\bsnu|+1)! ,\notag
\end{align}
where we used the Vandermonde convolution~$\sum_{\substack{|\bsm|=\ell\\ \bsm \le \bsnu}} \binom{\bsnu}{\bsm} = \binom{|\bsnu|}{\ell} =  \frac{|\bsnu|!}{\ell! (|\bsnu|-\ell)!}$.

For the second QoI in~\eqref{eq:QoIs} we use the optimality condition~\eqref{eq:gradient}, i.e.,~$u(\bssigma) = B^\ast q_1(\bssigma)$. Using the Cauchy--Schwarz inequality and the triangle inequality, and~\eqref{eq:analyticadjoint} yields
\begin{align}
    &\big|\partial^\bsnu_{\bssigma} \|B^\ast q_1(\bssigma)\|_{U_T}^2\big| \le \sum_{\bsm\le \bsnu} \binom{\bsnu}{\bsm} \| \partial^\bsm_{\bssigma} B^\ast q_1(\bssigma)\|_{U_T} \|\partial^{\bsnu-\bsm}_{\bssigma}  B^\ast q_1(\bssigma) \|_{U_T} \notag \\
    &\le \sum_{\bsm \le \bsnu} \binom{\bsnu}{\bsm} \|B\|_{\calL(U,H)} c_V c_\calG(T) |\bsm|! \bsb^\bsm \|(f , y_\circ, Q^\ast Qg, P^\ast P g_T)\|_{V'_T \times H \times V'_T \times H}  \notag\\
    &\qquad \qquad \times \|B\|_{\calL(U,H)} c_V c_\calG(T) |\bsnu -\bsm|! \bsb^{\bsnu-\bsm}\|(f , y_\circ, Q^\ast Qg, P^\ast P g_T)\|_{V'_T \times H \times V'_T \times H}  \notag\\
    &\le \|B\|_{\calL(U,H)}^2 c_V^2 c_\calG(T)^2 \|(f , y_\circ, Q^\ast Qg, P^\ast P g_T)\|_{V'_T \times H \times V'_T \times H}^2 \bsb^\bsnu \sum_{\ell = 0}^{|\bsnu|} |\ell|! (|\bsnu|-\ell)! \sum_{\substack{|\bsm|=\ell\\ \bsm \le \bsnu}} \binom{\bsnu}{\bsm}\notag\\
    &= \|B\|_{\calL(U,H)}^2 c_V^2 c_\calG(T)^2 \|(f , y_\circ, Q^\ast Qg, P^\ast P g_T)\|_{V'_T \times H \times V'_T \times H}^2 \bsb^\bsnu(|\bsnu|+1)!,\notag
\end{align}
where we used the Vandermonde convolution as in the steps for the first QoI.

In order to prove the bound for the third QoI in~\eqref{eq:QoIs}, we use the following estimate:
\begin{align}
    \|\partial_{\bssigma}^\bsnu y_{\bssigma}(T)\|_H^2 - \|\partial_{\bssigma}^\bsnu y_{\bssigma}(0)\|_H^2 &= \int_0^T \frac{\mathrm d}{\mathrm dt} \|\partial_{\bssigma}^\bsnu y_{\bssigma}(t)\|_H^2 \,\mathrm dt = 2\int_0^T \langle \frac{\mathrm d}{\mathrm dt} \partial_{\bssigma}^\bsnu y_{\bssigma}(t), \partial_{\bssigma}^\bsnu y_{\bssigma}(t) \rangle_{V',V}\mathrm dt \notag\\
    &\le 2 \|\frac{\mathrm d}{\mathrm dt} \partial^\bsnu_{\bssigma} y_{\bssigma}\|_{V'_T} \|\partial^\bsnu_{\bssigma} y_{\bssigma}\|_{V_T} \le  \|\partial^\bsnu_{\bssigma} y_{\bssigma}\|_{W_T(V,V')}^2. \notag
\end{align}
Together with~\eqref{eq:analyticstate}, this leads to
\begin{align}\label{eq:pdiff}
    \|\partial_{\bssigma}^\bsnu y_{\bssigma}(T)\|_H^2 &\le \|\partial^\bsnu_{\bssigma} y_{\bssigma}\|_{W_T(V,V')}^2 + \|\partial_{\bssigma}^\bsnu y_{\bssigma}(0)\|_H^2\notag\\ &\le (c_\calG(T)^2+1) (|\bsnu|! \bsb^\bsnu)^2 \|(f,y_\circ,Q^\ast Qg, P^\ast Pg_T)\|_{V'_T\times H \times V'_T \times H}^2,
\end{align}
where we used $\|\partial^\bsnu_{\bssigma} y_{\bssigma}(0)\|_H^2 \le (|\bsnu|! \bsb^\bsnu)^2 \|(f,y_\circ,Q^\ast Q g, P^\ast P g_T)\|_{V_T' \times H \times V_T' \times H}^2$ for all~$\bsnu \in \calF$.
Then, setting~$C_{\rm ter} :=  (c_\calG(T)^2+1)$, and by application of the Cauchy--Schwarz inequality as well as the triangle inequality, we get by~\eqref{eq:pdiff}
\begin{align}
    \big|\partial^\bsnu_{\bssigma}  \|P(y_{\bssigma}(T)-g_T)\|^2_{H}\big| 
    &\le \sum_{\bsm \le \bsnu} \binom{\bsnu}{\bsm} \| \partial^\bsm_{\bssigma} P(y_{\bssigma}(T) - g_T)\|_H \|\partial^{\bsnu-\bsm}_{\bssigma} P(y_{\bssigma}(T) - g_T)\|_H \notag\\ 
    &= \sum_{\bsm \le \bsnu} \binom{\bsnu}{\bsm} \|P \partial^\bsm_{\bssigma} (y_{\bssigma}(T) - g_T)\|_H \|P\partial^{\bsnu-\bsm}_{\bssigma} (y_{\bssigma}(T) - g_T)\|_H\notag \\ 
    &\le \|P\|_{\calL(H)}^2 \sum_{\bsm \le \bsnu} \binom{\bsnu}{\bsm}  \|\partial^\bsm_{\bssigma} (y_{\bssigma}(T) - g_T)\|_H \|\partial^{\bsnu-\bsm}_{\bssigma} (y_{\bssigma}(T) - g_T)\|_H. \notag
\end{align}
Since~$g_T$ is independent of~$\bssigma\in\fkS$, we have for all~$\bsnu \in \calF$ that~$\|\partial_{\bssigma}^\bsnu g_T\|_H \le \delta_{\bsnu, \boldsymbol{0}} |\bsnu|! \bsb^\bsnu \|g_T\|_{H} \le |\bsnu|! \bsb^\bsnu \|g_T\|_{H}$. Setting~$\tilde{C}_{\rm ter} := C_{\rm ter} \|(f,y_\circ,Q^\ast Q g, P^\ast P g_T)\|_{V_T' \times H \times V_T' \times H}$ we get by~\eqref{eq:pdiff}
\begin{align}
    &\big|\partial^\bsnu_{\bssigma}  \|P(y_{\bssigma}(T)-g_T)\|^2_{H}\big|\notag\\
    &\quad \le \|P\|_{\calL(H)}^2 \sum_{\bsm \le \bsnu} \binom{\bsnu}{\bsm}  (\|g_T\|_H + \tilde{C}_{\rm ter}) |\bsm|! \bsb^\bsm (\|g_T\|_H + \tilde{C}_{\rm ter}) |\bsnu-\bsm|! \bsb^{\bsnu-\bsm}\notag\\
    &\quad\le \|P\|_{\calL(H)}^2 (\|g_T\|_H + \tilde{C}_{\rm ter})^2 \sum_{\bsm \le \bsnu} \binom{\bsnu}{\bsm}   |\bsm|! \bsb^\bsm |\bsnu-\bsm|! \bsb^{\bsnu-\bsm} \notag\\
    &\quad\le \|P\|_{\calL(H)}^2 (\|g_T\|_H + \tilde{C}_{\rm ter})^2 (|\bsnu|+1)! \bsb^\bsnu,\notag
\end{align}
which proves the desired result.
\end{proof}

\section{Analytic dependence of the feedback law on the parameters}\label{sec:analyticFB}
It is well-known that the optimal control, minimizing~\eqref{eq:J1} subject to~\eqref{eq:psys}, can be written as a feedback law (depending on~$\bssigma$) applied to the optimal state, that is
\begin{align*}
    u(\bssigma) = K(\bssigma, y(\bssigma)).
\end{align*}
In order to apply QMC integration to approximate the mean-based feedback~\eqref{eq:FBgoal} in Sect.~\ref{sec:QMC}, we investigate the regularity of~$\bssigma \mapsto K(\bssigma)$. 
In fact, under the conditions of the previous section, we will show that the optimal feedback law depends analytically on the parameters.

In the presented parameterized linear quadratic optimal control problem the feedback is based on a parameter-dependent differential Riccati equation. In \cite{Ran1988} the authors investigate the analyticity of the solution to the algebraic Riccati equations with respect to perturbations in the system matrices. However, no bounds on the derivatives are provided, which are of central importance in our work.

We divide the following analysis into two cases: first we investigate the case~$g = g_T = f = 0$, and then we consider nontrivial target functions~$g$ and $g_T$ as well as nontrivial external forcing~$f$.

\subsection{Homogeneous constraint}\label{sec:homogeneous}
In the homogeneous case, with~$f=g=g_T=0$, the optimal feedback law is given as
\begin{align}\label{eq:linearfb}
    K_{\bssigma}(t) = -B^\ast \Pi_{\bssigma}(T-t),
\end{align}
where~$\Pi_{\bssigma}(t)$, for~$t\in(0,T)$, solves the operator-valued differential Riccati equation
\begin{align}
\dot{\Pi}_{\bssigma} &= \Pi_{\bssigma} \calA_{\bssigma} +\calA_{\bssigma}^\ast \Pi_{\bssigma} - \Pi_{\bssigma} B B^\ast \Pi_{\bssigma}+ Q^*Q,\qquad
\Pi_{\bssigma}(0) = P^*P.\label{eq:Riccati}
\end{align}

For our analysis, we will use the fact that for all~$\bssigma$ the optimal cost is given by
\begin{align}\label{eq:opticost}
    \frac12 \langle \Pi_{\bssigma}(T) y_\circ, y_\circ\rangle_H = \calJ(y_{\bssigma},u_{\bssigma}),
\end{align}
and that~$q_1(\bssigma;0) = \Pi_{\bssigma}(T) y_\circ$. First, we show that~$\partial_{\bssigma}^\bsnu q_1(\bssigma;0) \in H$, and thus~$\partial_{\bssigma}^\bsnu \Pi_{\bssigma}(T) y_\circ \in H$, for all~$\bsnu \in \calF$. To this end, we estimate, similar to~\eqref{eq:pdiff},
\begin{align}
\|\partial_{\bssigma}^\bsnu q_{1}(\bssigma;0)\|_H^2 &\le \|\partial^\bsnu_{\bssigma} q_{1}(\bssigma;\cdot)\|_{W_T(V,V')}^2 + \|\partial_{\bssigma}^\bsnu P^\ast P y(\bssigma;T)\|_H^2 
\le C (|\bsnu|!\bsb^\bsnu)^2, \notag
\end{align}
for some~$C>0$, where we use~$q_1(\bssigma;T) = - P^\ast P y(\bssigma;T)$, as well as~\eqref{eq:analyticadjoint} and~\eqref{eq:pdiff}.
Thus, we have~$\partial^\bsnu_{\bssigma} \langle \Pi_{\bssigma}(T) y_\circ, y_\circ\rangle_H = \langle \partial^\bsnu_{\bssigma} \Pi_{\bssigma}(T) y_\circ, y_\circ\rangle_H$ for all~$\bsnu \in \calF$. Taking the~$\bsnu$-th derivative of~\eqref{eq:opticost}, we find due to Lemma~\ref{lem:objregular},
\begin{align}
    \big|\langle \partial^\bsnu_{\bssigma} \Pi_{\bssigma}(T) y_\circ, y_\circ\rangle_H\big| 
    &\le C_4(T) (|\bsnu|+1)! \bsb^\bsnu,\quad \forall \bsnu \in \calF,\notag
\end{align}
where~$C_4(T)$ is defined in Lemma~\ref{lem:objregular}. In the case~$f=g=g_T=0$ we have 
\begin{align*}
    C_4(T) = \left(\|Q\|_{\calL(H)}^2 c_V^2 c_\calG(T)^2 + \|B\|_{\calL(U,H)}^2 c_V^2 c_\calG(T)^2 + \|P\|_{\calL(H)}^2 (2c_\calG(T)^2 +1)\right) \|y_\circ\|^2_H.
\end{align*}

Since~$\partial^\bsnu_{\bssigma} \Pi_{\bssigma}(T)$ is bounded, linear and self-adjoint, taking the supremum over all initial conditions~$y_\circ$ on the unit sphere in~$H$ leads to
\begin{align}
    \|\partial^\bsnu_{\bssigma} \Pi_{\bssigma}(T)\|_{\calL(H)} = \sup_{\|y_\circ\|_H= 1} |\langle \partial^\bsnu_{\bssigma} \Pi_{\bssigma}(T) y_\circ, y_\circ\rangle_H|.\notag
\end{align}
Thus, it holds that
\begin{align}\label{eq:boundPiT}
    \|\partial^\bsnu_{\bssigma} \Pi_{\bssigma}(T)\|_{\calL(H)} &\le \sup_{\|y_\circ\|_H = 1} \Big(C_4(T) (|\bsnu|+1)! \bsb^\bsnu\Big) = C_5(T) (|\bsnu|+1)! \bsb^\bsnu, 
\end{align}
with~$C_5(T) :=  \left(\|Q\|_{\calL(H)}^2 c_V^2 c_\calG(T)^2 + \|B\|_{\calL(U,H)}^2 c_V^2 c_\calG(T)^2 + \|P\|_{\calL(H)}^2 (2c_\calG(T)^2 +1)\right)$.

To bound~$\Pi_{\bssigma}(\tau)$ for~$\tau \in [0,T]$ consider~\eqref{eq:J1} with~$T$ replaced by~$\tau$. In view of the autonomy of~$\calA_{\bssigma}, B,P$, and~$Q$ the corresponding Riccati equation is~$\eqref{eq:Riccati}$ restricted to~$[0,\tau]$. Thus, we can replace~$T$ in~\eqref{eq:boundPiT} by any~$\tau \in [0,T]$ and then take the supremum over all~$\tau \in [0,T]$ leading us to
\begin{align}\label{eq:boundPi}
    \|\partial^\bsnu_{\bssigma} \Pi_{\bssigma}(\tau)\|_{\calL(H)} &\le \sup_{\tau \in [0,T]} \Big(C_5(\tau) (|\bsnu|+1)! \bsb^\bsnu\Big) = C_5(T) (|\bsnu|+1)! \bsb^\bsnu,
\end{align}
since~$C_5(T)$ is continuous and monotonically increasing in~$T$. 

We have thus shown the following result.
\begin{theorem}\label{thm:fbhomog}
    Under the assumptions of Theorem~\ref{thm:s_saddle_op}, the feedback law~\eqref{eq:linearfb} depends analytically on the parameters~$\bssigma \in \fkS$ with
    \begin{align*}
    \|\partial_{\bssigma}^\bsnu (-B^\ast \Pi_{\bssigma}(T-t))\|_{\calL(H,U)} \le \|B\|_{\calL(U,H)} C_5(T) (|\bsnu|+1)! \bsb^\bsnu \qquad \forall t\in[0,T],
\end{align*}
for all~$\bsnu \in \calF$.
\end{theorem}

\subsection{Nonhomogeneous case}\label{sec:nonhomogeneous}
Let us consider the nontrivial tracking problem with~$g \ne 0$ and~$g_T\ne 0$ in~\eqref{eq:J1}, as well as~$f \neq 0$ in~\eqref{eq:psys}. After the variable transformation~$x_{\bssigma} := y_{\bssigma} - g$ the nonhomogeneous term~$r_{\bssigma} := f(t) + \calA_{\bssigma} g(t) - \dot{g}(t)$ naturally arises in the state equation of tracking-type problems. We recall the following result from~\cite[Thm.~7.1, Part IV, Ch.~1]{bensoussan2007representation}, for which we assume the additional regularity~$f \in H_T$ and~$g \in W^{1,2}(0,T;H)\cap L^2(0,T;D(\calA_{\bssigma}))$.%
\begin{theorem}
Given~$\bssigma \in \fkS$, let~$\Pi_{\bssigma} \in C([0,T];\calL(H))$ denote the unique self-adjoint and nonnegative solution of~\eqref{eq:Riccati}. Then, there exists a unique minimizer~$(x_{\bssigma},u_{\bssigma})$ of~\eqref{eq:J1} subject to~\eqref{eq:psys}. This optimal pair satisfies, for~$t\in(0,T)$,
\begin{itemize}
\item[1.] $u_{\bssigma}$ is given in feedback form by
\begin{align}\label{eq:uopt}
 u_{\bssigma}(t) = - B^\ast \left( \Pi_{\bssigma}(T-t) x_{\bssigma}(t) + h_{\bssigma}(t) \right);
\end{align}
\item[2.] $x_{\bssigma}$ is the mild solution to the closed-loop system
\begin{align*}
\dot{ x}_{\bssigma}(t) &= \left(\calA_{\bssigma} - BB^\ast \Pi_{\bssigma}(T-t) \right) x_{\bssigma}(t) - BB^\ast h_{\bssigma}(t) + r_{\bssigma}(t),\qquad x(0) = x_\circ;
\end{align*}
where
\begin{align}
-\dot{h}_{\bssigma}(t) &= \left(\calA_{\bssigma}^\ast - \Pi_{\bssigma}(T-t)BB^\ast\right) h_{\bssigma}(t) + \Pi_{\bssigma}(T-t)r_{\bssigma}(t),\qquad
h_{\bssigma}(T) = 0;\label{eq:h}
\end{align}
\item[3.] the optimal cost is given by
\begin{equation*}
\begin{split}
\calJ(x, u) &= \frac{1}{2} \langle \Pi_{\bssigma}(T) x_\circ,x_\circ\rangle_{ H} + \langle h_{\bssigma}(0), x_\circ\rangle_{ H} \\
&\quad+ \int_0^T \left( \langle h_{\bssigma}(s), r_{\bssigma}(s) \rangle_{ H} - \frac{1}{2}\|B^\ast h_{\bssigma}(s)\|_{ U}^2\right) \mathrm ds.
\end{split}
\end{equation*}
\end{itemize}
\end{theorem}

We see from~\eqref{eq:uopt} that, for~$t \in [0,T]$, the optimal feedback~$K_{\bssigma}(t,\cdot): H \to U$ in the nonhomogeneous case is an affine function of the optimal state~$x_{\bssigma}(t)$. It is given by
\begin{align}
    K_{\bssigma}(t,\cdot) = -B^\ast \big(\Pi_{\bssigma}(T-t) \cdot + h_{\bssigma}(t)\big). \notag
\end{align}
The parametric regularity of the linear part~$-B^\ast\Pi_{\bssigma}$ is derived in Sect.~\ref{sec:homogeneous}. Thus, it remains to investigate the parametric regularity of~$h_{\bssigma}$.
For this purpose we make the following assumptions:
\begin{align}
    \begin{split}
        D(\calA_{\bssigma}) \text{ is independent of } \bssigma \in \fkS \text{ and } D(\calA_{\bssigma})= D(\calA^\ast_{\bssigma}) \text{ for all } \bssigma \in \fkS, 
    \label{eq:A4}
    \end{split}\\
    \begin{split}
        &\exists\,\text{a sequence of nonnegative numbers } \tilde\bsb \in \ell^p(\mathbb{N}) \text{ with }0<p\le 1\\
        &\text{such that }
        \|\partial_\bssigma^\bsnu \calA_\bssigma\|_{\calL(D(\calA),H)} \le \tilde{\bsb}^\bsnu \quad \forall \bsnu \in \calF \setminus\{\boldsymbol{0}\},
    \end{split}
    \label{eq:A5} \\
    \begin{split}
    \|\calA_{\bssigma}\|_{\calL(D(\calA),H)} \le \widetilde C_\calA \quad \forall \bssigma \in \fkS.
    \label{eq:A6} 
    \end{split}
\end{align}  
In view of~\eqref{eq:A4} we will denote~$D(\calA) = D(\calA_{\bssigma}) = D(\calA^\ast_{\bssigma})$ for all~$\bssigma \in \fkS$.

\begin{proposition}\label{prop:hbound}
    Let the assumptions~\eqref{A1},~\eqref{A2},~\eqref{A3},~\eqref{eq:A4},~\eqref{eq:A5}, and~\eqref{eq:A6} hold, and assume that there exists a sequence~$\tilde \bsb \in \ell^p(\bbN)$, with~$0<p\le1$, of nonnegative numbers, such that~$\|\partial_{\bssigma}^\bsnu \calA_{\bssigma}\|_{\calL(V,V')} \le \tilde \bsb^{\bsnu}$ for all~$\bsnu \in \calF \setminus \{\boldsymbol{0}\}$.
    Then, the solution~$h_{\bssigma}$ of~\eqref{eq:h} satisfies
    \begin{align}
        \|\partial^\bsnu_{\bssigma} h_{\bssigma}\|_{W_T^0(V,V')} \le  \frac{1}{2} (1+C)^{\max\{|\bsnu|-1,0\}}C^{\delta_{\bsnu,\mathbf 0}}(C+C^2)^{1-\delta_{\bsnu,\mathbf 0}}(|\bsnu|+2)! \boldsymbol b^{\bsnu}, \quad \forall \bsnu \in \calF,\notag
    \end{align}
    with~$b_j = \tilde{b}_j/\ln{2}$, for some constant~$C>0$ independent of~$\bssigma \in \fkS$.
\end{proposition}
\begin{proof}
Let us define~$W^0_{T}(V,V') := \{w \in W_T(V,V') \mid w(T) = 0\}$ and the parametric evolution operator~$D_{\bssigma} = -\frac{d}{dt} - (\calA_{\bssigma}^\ast - \Pi_{\bssigma}(T-t) B B^\ast)$ mapping from $W^0_{T}(V,V')$ to $V'_T$, so that we can write~\eqref{eq:h} as
\begin{align}\label{eq:h2}
    D_{\bssigma} h_{\bssigma} = \Pi_{\bssigma} r_{\bssigma}.
\end{align}
Following the arguments in the proof of~ Lemma~\ref{lem:apriori_unc} and using the uniform bound on~$\Pi_{\bssigma}$ (\eqref{eq:boundPi} with~$\bsnu=\boldsymbol{0}$), it can be shown that there is a constant~$c_D>0$ such that~$\|D_{\bssigma}^{-1}\|_{\calL(V_T', W^0_{T}(V,V'))}\le c_D$ for all~$\bssigma \in \fkS$. Moreover, for arbitrary~$w\in W^0_{T}(V,V')$ and for all~$\bsnu \in \calF \setminus \{\boldsymbol{0}\}$, we find
\begin{align}
    \|\partial^\bsnu_{\bssigma} D_{\bssigma} w\|_{V_T'} &= \|\partial^\bsnu_{\bssigma} \calA^\ast_{\bssigma} w - \partial^\bsnu_{\bssigma} \Pi_{\bssigma} BB^\ast w \|_{V_T'} \notag\\
    &\le\big( \|\partial^\bsnu_{\bssigma} \calA^\ast_{\bssigma}\|_{\calL(V,V')} + c_Hc_V\|\partial^\bsnu_{\bssigma} \Pi_{\bssigma} BB^\ast\|_{\calL(H)}\big) \|w \|_{W^0_{T}(V,V')}, \notag
\end{align}
where~$c_H$ and~$c_V$ are the embedding constants~$\|\cdot\|_{V'} \le c_H \|\cdot\|_H$ and~$\|\cdot\|_{H} \le c_V \|\cdot\|_V$, respectively. Recalling that~$\|\partial_{\bssigma}^\bsnu \calA_{\bssigma}\|_{\calL(V,V')} \le \tilde \bsb^{\bsnu}$ for all~$\bsnu \in \calF \setminus \{\boldsymbol{0}\}$, and by~\eqref{eq:boundPi} and the fact that~$\tilde b_j \le b_j$,~$j\ge1$, we have
\begin{align}
    \|\partial^\bsnu_{\bssigma} D_{\bssigma}\|_{\calL(W_T^0(V,V'),V_T')} &\le C \big( \tilde\bsb^\bsnu + (|\bsnu|+1)! \bsb^\bsnu \big) = C (|\bsnu|+2)! \bsb^\bsnu,\label{eq:partialD}
\end{align}
for some~$C>0$ independent of~$\bssigma \in \fkS$. Thus,~$D_{\bssigma}$ is a~$p$-analytic operator. However, in order to deduce the parametric regularity of the solution~$h_{\bssigma}$ of~\eqref{eq:h2}, we cannot directly apply Theorem~\ref{thm:solanalytic} since the right-hand side in~\eqref{eq:h2} depends on~$\bssigma \in \fkS$.

Instead, we will prove the result by induction on~$|\bsnu|$. For the base case,~$\bsnu = \boldsymbol{0}$, we estimate the right-hand side of~\eqref{eq:h2}: by~\eqref{eq:A6} we have for~$\bsnu=\boldsymbol{0}$ that~$\|\partial_{\bssigma}^\bsnu r_{\bssigma}\|_{H_T} \le \|f\|_{H_T} + \widetilde C_\calA \|g\|_{H_T} + \|\dot{g}\|_{H_T}$ for all~$\bssigma \in \fkS$. Together with the uniform boundedness of~$\Pi_{\bssigma}$ and~$D_{\bssigma}^{-1}$, the base step~$\|h_{\bssigma}\|_{W^0_T(V,V')} \le C$ follows. For the induction step let~$\bsnu \in \calF \setminus \{\boldsymbol{0}\}$.
By~\eqref{eq:A5} we have that~$\|\partial_{\bssigma}^\bsnu r_{\bssigma}\|_{H_T} = \|\partial_{\bssigma}^\bsnu \calA_{\bssigma} g\|_{H_T} \le \|\partial_{\bssigma}^\bsnu \calA_{\bssigma}\|_{\calL(D(\calA),H)} \|g\|_{H_T} \le  \|g\|_{H_T} \tilde\bsb^\bsnu$, and thus 
\begin{align}
\|\partial_{\bssigma}^\bsnu r_{\bssigma}\|_{H_T} \le C_6 \tilde\bsb^\bsnu, \notag
\end{align}
for some constant~$C_6>0$ independent of~$\bssigma\in \fkS$. In the following, let~$\bsnu \in \calF$. Then, by~\eqref{eq:boundPi} and the Leibniz product rule we have
\begin{align}
    \|\partial^\bsnu_{\bssigma} \Pi_{\bssigma}(T-t) r_{\bssigma}(t)\|_{H_T} &= \Big\| \sum_{\bsm \le \bsnu} \binom{\bsnu}{\bsm} \partial^\bsm \Pi_{\bssigma}(T-t) \partial^{\bsnu-\bsm} r_{\bssigma}(t)\Big\|_{H_T} \notag\\
    &\le C_5(T)C_6 \sum_{\bsm \le \bsnu} \binom{\bsnu}{\bsm} (|\bsm|+1)! \bsb^\bsm \tilde\bsb^{\bsnu-\bsm}\notag\\
    &= C_7(T) \bsb^\bsnu \sum_{\ell = 0}^{|\bsnu|} (\ell+1)!  \frac{|\bsnu|!}{\ell! (|\bsnu|-\ell)!} \le C_7(T) \bsb^\bsnu (|\bsnu|+2)!,\label{eq:partialPir}
\end{align}
where~$C_7(T) = C_5(T)C_6$ and~$\tilde b_j \le b_j$,~$j\ge 1$.

Applying~$\partial^\bsnu_{\bssigma}$ to~\eqref{eq:h2} leads by the Leibniz product rule to
\begin{align}
    \sum_{\bsm \le \bsnu} \binom{\bsnu}{\bsm} \partial_{\bssigma}^{\bsm} D_{\bssigma} \partial_{\bssigma}^{\bsnu-\bsm} h_{\bssigma} = \sum_{\bsm \le \bsnu} \binom{\bsnu}{\bsm} \partial^\bsm \Pi_{\bssigma}(T-t) \partial^{\bsnu-\bsm} r_{\bssigma}(t),\notag
\end{align}
provided that the derivatives of~$h_{\bssigma}$ exist. Their existence can be shown as follows: separating out the~$\bsm = \boldsymbol{0}$ term on the left-hand side leads to the recurrence relation
\begin{align}\label{eq:Dbsnuh}
    D_{\bssigma} \partial^\bsnu_{\bssigma} h_{\bssigma} = - \sum_{\substack{\bsm \le \bsnu \\\bsm\ne \boldsymbol{0}}} \binom{\bsnu}{\bsm} \partial_{\bssigma}^{\bsm} D_{\bssigma} \partial_{\bssigma}^{\bsnu-\bsm} h_{\bssigma} + \sum_{\bsm \le \bsnu} \binom{\bsnu}{\bsm} \partial^\bsm \Pi_{\bssigma}(T-t) \partial^{\bsnu-\bsm} r_{\bssigma}(t).
\end{align}
The existence of~$\partial^\bsnu_{\bssigma} h_{\bssigma}$ follows from~\eqref{eq:Dbsnuh} by induction on~$|\bsnu|$: for~$\bsnu= \boldsymbol{0}$ it follows from~\eqref{eq:h2}. Then, with~\eqref{eq:partialD} and~\eqref{eq:partialPir} the existence of~$\partial^\bsnu_{\bssigma} h_{\bssigma}$ follows from~\eqref{eq:Dbsnuh} for any~$\bsnu \in \calF$ assuming that it has been shown for all multi-indices with order less than~$|\bsnu|$. 
 
Furthermore,~\eqref{eq:Dbsnuh} leads to 
\begin{align}
    \|\partial^\bsnu_{\bssigma} h_{\bssigma}\|_{W^0_{T}(V,V')} &\le \sum_{\substack{\bsm \le \bsnu \\\bsm\ne \boldsymbol{0}}} \binom{\bsnu}{\bsm} \|D_{\bssigma}^{-1} \partial_{\bssigma}^{\bsm} D_{\bssigma} \partial_{\bssigma}^{\bsnu-\bsm} h_{\bssigma}\|_{W^0_T(V,V')}\notag\\&\quad + \sum_{\bsm \le \bsnu} \binom{\bsnu}{\bsm} \|D_{\bssigma}^{-1} \partial^\bsm \Pi_{\bssigma}(T-t) \partial^{\bsnu-\bsm} r_{\bssigma}(t)\|_{W^0_T(V,V')}\notag\\
    &\le \hat C \sum_{\substack{\bsm \le \bsnu \\\bsm\ne \boldsymbol{0}}} \binom{\bsnu}{\bsm} (|\bsm|+2)! \bsb^{\bsm} \|\partial^{\bsnu-\bsm}_{\bssigma} h_{\bssigma}\|_{W^0_T(V,V')} + \hat C (|\bsnu|+2)! \bsb^\bsnu,\notag
\end{align}
for some~$\hat C>0$ independent of~$\bssigma \in \fkS$. Then, setting~$C = 2 \hat C $, by Lemma~\ref{lem:recursion} in Appendix B we have
\begin{align}
    \|\partial^\bsnu_{\bssigma} h_{\bssigma}\|_{W_T^0(V,V')} \le \frac{1}{2} (1+C)^{\max\{|\bsnu|-1,0\}}C^{\delta_{\bsnu,\mathbf 0}}(C+C^2)^{1-\delta_{\bsnu,\mathbf 0}}(|\bsnu|+2)! \boldsymbol b^{\bsnu}, \notag
\end{align}
which proves the desired result for all~$\bsnu \in \calF$.
\end{proof}

The regularity results for the feedback are summarized in the following theorem. The result directly follows by combining the embedding~$\max_{t \in [0,T]} \|\partial^\bsnu_{\bssigma} h_{\bssigma}(t)\|_H \le C(T)\|\partial^\bsnu_{\bssigma} h_{\bssigma}\|_{W_T(V,V')}$ with Proposition~\ref{prop:hbound} and Theorem~\ref{thm:fbhomog}.
\begin{theorem}\label{thm:combreg}
    Under the conditions of Proposition~\ref{prop:hbound} we have
    \begin{align*}
        \sup_{t \in [0,T]} \left( \|\partial^\bsnu_{\bssigma} (-B^\ast \Pi_{\bssigma}(T-t))\|_{\calL(H,U)} + \|\partial^\bsnu_{\bssigma} (-B^\ast h_{\bssigma}(t))\|_U \right) \le \|B\|_{\calL(H,U)} \,\widetilde{C} \,(|\bsnu|+2)!\, \bsb^\bsnu ,
    \end{align*}
    for all~$\bsnu \in \calF$ and a constant~$\widetilde{C}>0$ depending on~$T>0$, but independent of~$\bssigma \in \fkS$.
\end{theorem}
This regularity result is used in Sect.~\ref{sec:QMC} for the error analysis of the QMC approximation of the mean-based feedback~\eqref{eq:FBgoal}.

\section{QMC approximation of the feedback $K$}
\label{sec:QMC}

In numerical practice~\eqref{eq:FBgoal} signals the necessity of approximately evaluating integrals of the form
\begin{align*}
\int_{[-1/2,1/2]^{\bbN}} \scrK_{\bssigma}(t) \rd \bssigma,
\end{align*}
for fixed $t$, where
\begin{align*}
    \scrK_{\bssigma}(t) = -B^\ast \Pi_{\bssigma}    \in Z = \calL(H,U), \qquad
    \text{or} \qquad
    \scrK_{\bssigma}(t) = -B^\ast h_{\bssigma} \in Z= U.
\end{align*}

For this purpose, in the following section we investigate the potential of QMC methods for the feedback optimal control problem of minimizing~\eqref{eq:J1} subject to~\eqref{eq:psys}.

The QMC approximation of integrals with Banach space-valued integrands was first studied in~\cite{guth2022parabolic}, where the error analysis is presented for randomly shifted lattice rules. Higher-order QMC methods have been studied in~\cite{LongoSchwabStein} in a Hilbert space setting. In order to use higher-order QMC rules for the approximation of integrals over the feedback law, we provide a novel error analysis for higher-order QMC rules in separable Banach spaces.

Recall that the state space~$H$ is a separable Hilbert space.
This separability assumption allows to guarantee that~$\calL(H,U)$ is a separable Banach space~\cite[Prop.~3.13]{vomende}, since in our case~$U$ is chosen to be finite-dimensional. Furthermore, by Theorem~\ref{thm:combreg} the mapping~$\bssigma \mapsto \scrK_{\bssigma}$ is continuous. Thus,~$\bssigma \mapsto G(\scrK_{\bssigma})$ is continuous for all~$G\in Z'$, and in particular it is measurable. By Pettis' Theorem ~$\scrK_{\bssigma}$ is strongly measurable. Moreover, the upper bound in Thm.~\ref{thm:combreg} implies with Bochner's Theorem that~$\scrK_{\bssigma}$ is integrable over~$\fkS = [-1/2,1/2]^{\mathbb{N}}$.

Our strategy for approximation is to truncate the domain of $\bssigma$ to the finite-dimensional domain $[-1/2,1/2]^s$ first, and then, after some transformations, apply a quasi-Monte Carlo (QMC) rule to the truncated integral. 
Indeed, for $s\in\{1,\ldots,d\}$, we write $(\bssigma_s,\bszero):=(\sigma_1,\ldots,\sigma_s,0,0,\ldots)$, and $\scrK_{\bssigma, s}:=\scrK_{\bssigma} ((\bssigma_s,\bszero))$. 
With this strategy, the total approximation error consists of:
\begin{align*}
    \underbrace{\left\| \int_{[-1/2,1/2]^{\bbN}} (\scrK_{\bssigma}(t)-\scrK_{\bssigma,s}(t)) \rd \bssigma \right\|_{Z}}_{\text{dimension truncation error}}
    & + \underbrace{\left\|\int_{[-1/2,1/2]^{s}} \scrK_{\bssigma,s}(t) \rd \bssigma - \frac{1}{N} \sum_{k= 0}^{N-1} \scrK_{\bsx_k} \right\|_Z}_{\text{QMC error}},
\end{align*}
where the points~$\bsx_0,\ldots,\bsx_{N-1}\in [-1/2,1/2]^s$ are QMC points chosen as described in the remainder of this section. To bound the dimension truncation error let us assume that, for a.e.~$\bssigma \in [-1/2,1/2]^\bbN$ and for all~$t \in [0,T]$, we have that
\begin{equation}\label{eq:limit_K_s}
\|\scrK_{\bssigma}(t)-\scrK_{\bssigma,s}(t) \|_{Z}\to 0 \quad \mbox{as}\quad s\to\infty.
\end{equation}
This assumption is addressed in Appendix~\ref{sec:appendixdimtrunc}.

Due to Theorem~\ref{thm:combreg} and \eqref{eq:limit_K_s}, we can apply \cite[Thm. 4.3]{GuthKaa23}, which yields 
\[
 \left\| \int_{[-1/2,1/2]^{\bbN}} (\scrK_{\bssigma}(t)-\scrK_{\bssigma,s}(t)) \rd \bssigma \right\|_{Z}
 \le C\, s^{-\frac{2}{p}+1},
\]
where $C$ is a positive constant independent of $s$. 
The remainder of this section is devoted to the development of QMC rules to approximate
\begin{equation*}
\int_{[-1/2,1/2]^s} \scrK_{\bssigma,s}(t) \rd \bssigma_s,
\end{equation*}
where~$\rd \bssigma_s = \bigotimes_{j = 1}^s \rd \sigma_j$ is a ``truncated'' version of~$\rd \bssigma$.

Next, we shall use a result that was essentially shown in \cite[Thm. 6.5]{guth2022parabolic}. We state this result for 
completeness.

\begin{theorem}\label{thm:GKKSS_Banach}
Let $\calW_s$ be a Banach space of functions $F:[-1/2,1/2]^s\to \bbR$, which is 
continuously embedded in the space of continuous functions. Consider an $N$-point QMC rule 
with integration nodes $\bsx_0, \ldots,\bsx_{N-1} \in [-1/2,1/2]^s$, given by
\[
  Q_{N,s}(F):= \frac{1}{N}\sum_{k=0}^{N-1} F(\bsx_k).
\]
Furthermore, we define the worst case error of integration using $Q_{N,s}$ in $\calW_s$ by
\[
  e^{\rm wor} (Q_{N,s}, \calW_s) := \sup_{\substack{F\in \calW_s\\ \| F\|_{\calW_s} \le 1}} 
  \left|\int_{[-1/2,1/2]^s} F(\bsx) \rd \bsx - Q_{N,s} (F)\right|.
\]
Assume that $Z$ is a separable Banach space and that $Z'$ is its dual space. Moreover, let $k$ be a continuous mapping that maps any $\bssigma \in [-1/2,1/2]^s$ to some $k(\bssigma)\in Z$. Then,
\[
 \left\|\int_{[-1/2,1/2]^s} k(\bssigma)\rd\bssigma - 
 \frac{1}{N}\sum_{k=0}^{N-1} k(\bsx_k)\right\|_Z \le 
 e^{\rm wor} (Q_{N,s}, \calW_s) \, \sup_{\substack{G\in Z'\\ \| G \|_{Z'}\le 1}} \| G(k) \|_{\calW_s} .
\]
\end{theorem}
If we choose $k=\scrK_{\sigma,s}(t)$, and $Z\in \{U,\calL(H,U)\}$ in Theorem \ref{thm:GKKSS_Banach}, we can apply this result to derive error bounds for approximating 
\[
 \int_{[-1/2,1/2]^{s}} \scrK_{\bssigma,s}(t) \rd \bssigma .
\]
To this end, we make a particular choice of the space $\calW_s$ in Theorem \ref{thm:GKKSS_Banach}. Actually, as we will show, there are different ways of choosing $\calW_s$ and a suitable QMC rule $Q_{N,s}$ tailored to $\calW_s$, leading to different variants of error estimates. We go through several cases separately below, but before would like to explain the general idea of \textit{weighted function spaces} as introduced by Sloan and Wo\'{z}niakowski in \cite{slowo1998weights}. 

The motivation for the use of weighted spaces results from applications, where different coordinates or different groups of 
coordinates may have different influence on a multivariate problem. In our case, this is reflected in the upper bound in Theorem~\ref{thm:combreg}; indeed, in this bound the terms $\bsb^\bsnu$ differ, depending on which $\partial_{\bssigma}^\bsnu h_{\bssigma}$ is considered. Consequently, as we will see below (see~\eqref{eq:spod_weights} and~\eqref{eq:pod_weights}), these terms will have a crucial role in the choice of our \textit{weights}, which are nonnegative real numbers $\gamma_{\fraku}$, one for each set $\fraku \subseteq \{1,\ldots,s\}$. Intuitively speaking, the number $\gamma_{\fraku}$ models the influence 
of the variables with indices in $\fraku$. Larger values of $\gamma_{\fraku}$ mean more influence, smaller values less influence. Formally, we
set $\gamma_{\emptyset}=1$, and we write $\bsgamma=\{\gamma_{\fraku}\}_{\fraku\subseteq \{1,\ldots, s  \}}$. These weights are included in the norms of the function spaces, which are, in this context so-called unanchored weighted Sobolev spaces. We will follow a strategy first outlined in~\cite{DKLNS14}, where we essentially aim to choose the weights in a way that the norms of the elements of the function space are bounded by a constant and the worst-case error of integration is minimized. This yields the particular choices ~\eqref{eq:spod_weights} and~\eqref{eq:pod_weights}, respectively. As it turns out, in this way the problems under consideration in this manuscript may even lose the curse of dimensionality, provided that suitable conditions on the parameters involved hold. We refer to the sections below for details.

The aforementioned unanchored weighted Sobolev spaces of smoothness $\alpha$ are defined as follows.

Let~$\alpha \in \bbN$, $1\le r\le \infty$, and let~$1\le q \le \infty$. We may choose~$\calW_s$ in Theorem \ref{thm:GKKSS_Banach} as the Sobolev space $\calW_{s,\alpha,\bsgamma,q,r}$, 
with positive coordinate weights $\bsgamma=(\gamma_{\fraku})_{\fraku\subseteq \{1,\ldots,s\}}$, and with norm
\begin{eqnarray*}
\lefteqn{\|F\|_{\calW_{s,\bsalpha,\bsgamma,q,r}}^r :=
\sum_{\fraku\subseteq \{1,\ldots,s\}} \bigg( \frac{1}{\gamma_{\fraku}^q} \sum_{\frak{v}\subseteq \fraku} \sum_{\boldsymbol{\tau}_{\fraku \setminus \frak{v}} \in \{1,\ldots,\alpha\}^{\mid \fraku \setminus \frak{v}\mid}}}\\
&\quad&\int_{[-1/2,1/2]^{|\frak{v}|}} 
\bigg| \int_{[-1/2,1/2]^{s-|\frak{v}|}} 
\frac{\partial^{(\bsalpha_{\frak{v}},\boldsymbol{\tau}_{\fraku \setminus \frak{v}},\boldsymbol{0})}}{\partial \bsy_{\fraku}} 
F(\bsy_{\fraku};\bsy_{\{1:s\}\setminus \fraku}) \rd \bsy_{\{1:s\}\setminus \frak{v}}
\bigg|^q \rd \bsy_{\frak{v}} \bigg)^{\frac{r}{q}},\notag
\end{eqnarray*}
with the obvious adaptions if $r=\infty$.
Here, by~$(\bsalpha_{\frak{v}},\boldsymbol{\tau}_{\fraku \setminus \frak{v}},\boldsymbol{0})$ we denote a sequence~$\bsnu$ with components~$\nu_j = \alpha$ for~$j\in \frak{v}$,~$\nu_j = \tau_j$ for~$j \in \fraku \setminus \frak{v}$, and~$\nu_j=0$ for~$j \notin \fraku$.

\medskip

For technical reasons, we also define another unanchored Sobolev space $\calV_{s,\alpha,\bsgamma}$ of smoothness $\alpha\in\bbN$, with the norm  
\begin{align*}
\|F\|_{\calV_{s,\bsalpha,\bsgamma}}^2 :=& \sum_{\fraku\subseteq \{1,\ldots,s\}}  \frac{1}{\gamma_{\fraku}^2} \sum_{\frak{v}\subseteq \fraku} \sum_{\boldsymbol{\tau}_{\fraku \setminus \frak{v}} \in \{1,\ldots,\alpha-1\}^{\mid \fraku \setminus \frak{v}\mid}} \\
&\quad\int_{[-1/2,1/2]^{|\frak{v}|}} 
\bigg| \int_{[-1/2,1/2]^{s-|\frak{v}|}} 
\frac{\partial^{(\bsalpha_{\frak{v}},\boldsymbol{\tau}_{\fraku \setminus \frak{v}},\boldsymbol{0})}}{\partial \bsy_{\fraku}} 
F(\bsy_{\fraku};\bsy_{\{1:s\}\setminus \fraku}) \rd \bsy_{\{1:s\}\setminus \frak{v}}
\bigg|^2 \rd \bsy_{\frak{v}} ,\notag
\end{align*}
with the following adjustment for the case $\alpha=1$: 
in this case, all summands in the second sum for $\frak{v}\subsetneq\fraku$ disappear. If $\frak{v}=\fraku$, then $\boldsymbol{\tau}_{\fraku \setminus \frak{v}}$ is the ``void'' vector with no components such that the corresponding summand in the second sum remains. 

Essentially, the norm $\|\cdot\|_{\calV_{s,\bsalpha,\bsgamma}}$ is similar to $\|\cdot\|_{\calW_{s,\alpha,\bsgamma,2,2}}$, but with the difference 
that the summation range for $\boldsymbol{\tau}_{\fraku \setminus \frak{v}}$ is $\{1,\ldots,\alpha-1\}^{\mid \fraku \setminus \frak{v}\mid}$ instead of $\{1,\ldots,\alpha\}^{\mid \fraku \setminus \frak{v}\mid}$.

Note that for any $F\in \calW_{s,\alpha,\bsgamma,2,2}$ we have $\|F\|_{\calV_{s,\alpha,\bsgamma}} \le \|F\|_{\calW_{s,\alpha,\bsgamma,2,2}}$, and this implies, 
for any QMC rule $Q_{N,s}$,
\begin{equation}\label{eq:norm_WV_ineq}
e^{\rm wor}(Q_{N,s},\calW_{s,\alpha,\bsgamma,2,2})
\le e^{\rm wor}(Q_{N,s},\calV_{s,\alpha,\bsgamma}).
\end{equation}
Thus, any upper bound on the worst-case error of a QMC rule $Q_{N,s}$ in $\calV_{s,\bsalpha,\bsgamma}$ is also an upper bound on the worst-case error of $Q_{N,s}$ in $\calW_{s,\bsalpha,\bsgamma,2,2}$.

We can then in principle proceed as in the proof of \cite[Theorem 6.6]{guth2022parabolic} as follows. For dealing with $\calW_{s,\bsalpha,\bsgamma,q,r}$, we
distinguish the cases $r<\infty$, and $r=\infty$, and we assume $q<\infty$ for simplicity.

For $r=\infty$, and $F\in \calW_{s,\bsalpha,\bsgamma,q,\infty}$, 
we get
\begin{align*}
\|F\|_{ \calW_{s,\alpha,\bsgamma,q,\infty}} 
\le& 
\max_{\fraku\subseteq \{1,\ldots,s\}}  \frac{1}{\gamma_{\fraku}} \bigg(\sum_{\frak{v}\subseteq \fraku} \sum_{\boldsymbol{\tau}_{\fraku \setminus \frak{v}} \in \{1,\ldots,\alpha\}^{\mid \fraku \setminus \frak{v}\mid}} 
\int_{[-1/2,1/2]^{s}} 
\bigg| 
\frac{\partial^{(\bsalpha_{\frak{v}},\boldsymbol{\tau}_{\fraku \setminus \frak{v}},\boldsymbol{0})}}{\partial \bsy_{\fraku}} 
F(\bsy )\bigg|^q  
\rd \bsy  \bigg)^{\frac{1}{q}},
\end{align*}
and for $r<\infty$ and $F\in \calW_{s,\bsalpha,\bsgamma,q,r}$, 
\begin{align*}
\|F\|_{ \calW_{s,\bsalpha,\bsgamma,q,r}}^r 
\le& 
\sum_{\fraku\subseteq \{1,\ldots,s\}}  \bigg(\frac{1}{\gamma_{\fraku}}\sum_{\frak{v}\subseteq \fraku} \sum_{\boldsymbol{\tau}_{\fraku \setminus \frak{v}} \in \{1,\ldots,\alpha\}^{\mid \fraku \setminus \frak{v}\mid}} 
\int_{[-1/2,1/2]^{s}} 
\bigg| 
\frac{\partial^{(\bsalpha_{\frak{v}},\boldsymbol{\tau}_{\fraku \setminus \frak{v}},\boldsymbol{0})}}{\partial \bsy_{\fraku}} 
F(\bsy )\bigg|^q  
\rd \bsy  \bigg)^{\frac{r}{q}}.
\end{align*}
In both cases, we can replace $F(\bsy)$ by $G(\scrK_{\bssigma,s}(t))$, integrate 
with respect to $\bssigma$ instead of $\bsy$, and obtain
for $G\in Z'$ with $\|G\|_{Z'}\le 1$,
\begin{eqnarray*}
    \int_{[-1/2,1/2]^{s}} 
    \bigg|     \frac{\partial^{(\bsalpha_{\frak{v}},\boldsymbol{\tau}_{\fraku \setminus \frak{v}},\boldsymbol{0})}}{\partial \bssigma_{\fraku}} 
    G(\scrK_{\bssigma,s}(t))\bigg|^q  
    \rd \bssigma
    &=&
    \int_{[-1/2,1/2]^{s}} 
    \bigg| 
G\left(\frac{\partial^{(\bsalpha_{\frak{v}},\boldsymbol{\tau}_{\fraku \setminus \frak{v}},\boldsymbol{0})}}{\partial \bssigma_{\fraku}} 
    \scrK_{\bssigma,s}(t)\right)\bigg|^q  
    \rd \bssigma \\
    &\le&
    \int_{[-1/2,1/2]^{s}} 
    \|G\|_{Z'}^q \, \left\| \frac{\partial^{(\bsalpha_{\frak{v}},\boldsymbol{\tau}_{\fraku \setminus \frak{v}},\boldsymbol{0})}}{\partial \bssigma_{\fraku}}  \scrK_{\bssigma,s}(t)\right\|_Z^q 
    \rd \bssigma\\
    &=&
    \int_{[-1/2,1/2]^{s}} 
    \left\| \frac{\partial^{(\bsalpha_{\frak{v}},\boldsymbol{\tau}_{\fraku \setminus \frak{v}},\boldsymbol{0})}}{\partial \bssigma_{\fraku}} \scrK_{\bssigma,s}(t)\right\|_Z^q 
    \rd \bssigma\\ \\
    &\le&\left(\widetilde{C}\,
     (|(\bsalpha_{\frak{v}},\boldsymbol{\tau}_{\fraku\setminus\frak{v}},\boldsymbol{0})|+2)! \, 
     \bsb_{\{1:s\}}^{(\bsalpha_{\frak{v}},\boldsymbol{\tau}_{\fraku\setminus\frak{v}},\boldsymbol{0})}
    \right)^q,
\end{eqnarray*} 
where we used Theorem~\ref{thm:combreg} in the last step. Thus, 
\begin{eqnarray}
\lefteqn{\sup_{\substack{G\in Z'\\ \| G\|_{Z'}\le 1}} \|G(\scrK_{\bssigma,s}(t))\|_{ \calW_{s,\bsalpha,\bsgamma,q,\infty}}}\notag\\ 
&\le&
\max_{\fraku\subseteq \{1,\ldots,s\}}  \frac{\widetilde{C}}{\gamma_{\fraku}} \bigg(\sum_{\frak{v}\subseteq \fraku} \sum_{\boldsymbol{\tau}_{\fraku \setminus \frak{v}} \in \{1,\ldots,\alpha\}^{\mid \fraku \setminus \frak{v}\mid}} 
 \left(
     (|(\bsalpha_{\frak{v}},\boldsymbol{\tau}_{\fraku\setminus\frak{v}},\boldsymbol{0})|+2)! \, 
     \bsb_{\{1:s\}}^{(\bsalpha_{\frak{v}},\boldsymbol{\tau}_{\fraku\setminus\frak{v}},\boldsymbol{0})}
    \right)^q \bigg)^{\frac{1}{q}}\notag\\
&\le&
\max_{\fraku\subseteq \{1,\ldots,s\}}  \frac{\widetilde{C}}{\gamma_{\fraku}} \sum_{\frak{v}\subseteq \fraku} \sum_{\boldsymbol{\tau}_{\fraku \setminus \frak{v}} \in \{1,\ldots,\alpha\}^{\mid \fraku \setminus \frak{v}\mid}} 
     (|(\bsalpha_{\frak{v}},\boldsymbol{\tau}_{\fraku\setminus\frak{v}},\boldsymbol{0})|+2)! \, 
     \bsb_{\{1:s\}}^{(\bsalpha_{\frak{v}},\boldsymbol{\tau}_{\fraku\setminus\frak{v}},\boldsymbol{0})}\notag\\
    &\le&
    \max_{\fraku\subseteq \{1,\ldots,s\}}  \frac{\widetilde{C}}{\gamma_{\fraku}}
    \sum_{\bsnu_\fraku \in \{1:\alpha\}^{|\fraku|}} (|\bsnu_\fraku|+2)! \prod_{j\in \fraku} (2^{\delta(\nu_j,\alpha)} b_j^{\nu_j}),\label{eq:sup_bound_infinite}
\end{eqnarray}
where~$\delta(\nu_j,\alpha)$ equals~$1$ if~$\nu_j = \alpha$, and is zero otherwise.

Analogously,
\begin{equation}\label{eq:sup_bound_finite}
\sup_{\substack{G\in Z'\\ \| G\|_{Z'}\le 1}}\|G(\scrK_{\bssigma,s}(t))\|_{ \calW_{s,\alpha,\bsgamma,q,r}}^r
\le
\sum_{\fraku\subseteq \{1,\ldots,s\}} \Bigg(\frac{\widetilde{C}}{\gamma_{\fraku}} 
    \sum_{\bsnu_\fraku \in \{1:\alpha\}^{|\fraku|}} (|\bsnu_\fraku|+2)! \prod_{j\in \fraku} (2^{\delta(\nu_j,\alpha)} b_j)^{\nu_j})\Bigg)^r.
\end{equation}
Reconsidering the above estimates implies that~\eqref{eq:sup_bound_finite} remains unchanged for~$q=\infty$.

We will now apply Theorem \ref{thm:GKKSS_Banach} for different choices of $\alpha$, $r$, and $q$, respectively, and outline how we can use well chosen QMC rules in order to obtain a suitable error bound. 

\subsection{Randomly shifted lattice rules for smoothness one}\label{sec:random_shift}
In this section, we consider the weighted Sobolev space of smoothness one, 
$\calV_{s,1,\bsgamma,}$ as introduced above, 
with norm
\[
\|F\|_{\calV_{s,1,\bsgamma}}^2:= \sum_{\fraku\subseteq \{1:s\}} \frac{1}{\gamma_{\fraku}^2}
\int_{[-1/2,1/2]^{|\fraku|}} 
\left| \int_{[-1/2,1/2]^{s-|\fraku|}} 
\frac{\partial^{|\fraku |}}{\partial \bsy_{\fraku}} 
F(\bsy_{\fraku};\bsy_{\{1:s\}\setminus \fraku}) \rd \bsy_{\{1:s\}\setminus \fraku}
\right|^2 \rd \bsy_{\fraku}.
\]
The space $\calV_{s,1,\bsgamma}$ is frequently considered in literature on lattice rules for multivariate integration.
Indeed, lattice point sets are highly-structured integration node sets in $[0,1)^s$. A sub-class of these are \textit{rank-1 lattice point sets}, essentially based on one \emph{generating vector} $\bsz=(z_1,\ldots,z_s)$. They are among the most prominent QMC integration node sets 
(see, e.g., \cite{dick2022lattices, nied1992siam, sloan1994lattices}). 
For a natural number $N \in \bbN$ and an integer generating vector $\bsz \in \{1,2,\dots,N-1\}^s$, rank-1 lattice points are of the form
\[  
 \bsx_k := \left\lbrace \frac{k}{N}\bsz\right\rbrace = \left( \left\{ \frac{k z_1}{N}\right\} , \ldots , 
 \left\{ \frac{k z_s}{N}\right\} \right)\in [0, 1)^s,\qquad\text{for } k = 0, 1, \ldots, N-1,
\] 
and the point set $\bsx_0,\bsx_1,\ldots,\bsx_{N-1}$ is sometimes denoted by $P(\bsz,N)$. For $\bsx \in\bbR^s$ we apply $\{\cdot\}$ component-wise, 
where $\{x\}=x-\lfloor x \rfloor$ is the fractional part of $x\in\bbR$. A QMC rule using $P(\bsz,N)$ as integration nodes 
is called a \emph{(rank-1) lattice rule}. Note that, given $N$ and $s$, 
a rank-1 lattice rule is completely determined by the choice of the generating vector $\bsz$. However, not every choice of $\bsz$ 
yields a rank-1 lattice rule with good quality for approximating the integral. The standard procedure to choose a good $\bsz$ is the so-called \textit{component-by-component construction} (or CBC construction, for short); a CBC construction 
is a greedy algorithm to choose the components of $\bsz$ one after another. Fast implementations of CBC constructions exist, many of them due to Cools and Nuyens (see, e.g., \cite{cools_nuyens}) and again \cite{dick2022lattices} for an overview.

It is sometimes useful to introduce a random element in the application 
of lattice rules. One way of doing so is to 
modify the lattice integration rule $\bsx_k$, $k=0,1,\ldots,N-1$ to $\{\bsx_k + \Delta\}$, 
$k=0,1,\ldots,N-1$, where $\Delta$ is chosen according to a uniform distribution on $[0,1]^d$, and where $\{\cdot\}$ again denotes the fractional part. In this way, we obtain a \textit{randomly shifted lattice rule} (all points are shifted by the same random $\Delta$). Note that by 
fixing a lattice point set and generating $M$ independent realizations of a random shift, 
one can study statistical properties of the corresponding estimators for the integral to be numerically computed.

It is known (see, e.g., \cite[Chapter 7]{dick2022lattices}) that we can effectively construct (by a CBC algorithm) a randomly shifted lattice rule with points $\bsx_0,\ldots,\bsx_{N-1}\in [-1/2,1/2]^s$ (shifted by a random $\bsDelta$) such that
\[
  \bbE_{\bsDelta}([e^{\rm wor}(Q_{N,s},\calV_{s,1,\bsgamma})]^2)\le 
  \left(\frac{1}{(\phi_{\rm tot}(N))}\sum_{\emptyset \neq \fraku \subseteq \{1:s\}}
  \gamma_{\fraku}^{2\lambda} 
  \left(\frac{2\zeta(2\lambda)}{(2\pi^2)^\lambda}\right)^{|\fraku|}\right)^{\frac{1}{\lambda}},
\]
for all $\lambda\in (1/2,1]$, where $\phi_{\rm tot}$ is Euler's totient function. By Equation \eqref{eq:norm_WV_ineq}, we then also obtain
\[
  \bbE_{\bsDelta}([e^{\rm wor}(Q_{N,s},\calW_{s,1,\bsgamma,2,2})]^2)\le 
  \left(\frac{1}{(\phi_{\rm tot}(N))}\sum_{\emptyset \neq \fraku \subseteq \{1:s\}}
  \gamma_{\fraku}^{2\lambda} 
  \left(\frac{2\zeta(2\lambda)}{(2\pi^2)^\lambda}\right)^{|\fraku|}\right)^{\frac{1}{\lambda}}.
\]
Next, we combine the latter inequality with \eqref{eq:sup_bound_finite}, and slightly modify Theorem \ref{thm:GKKSS_Banach} (see \cite[Theorem 6.6]{guth2022parabolic}) to obtain
\begin{equation}\label{eq:error_1_random}
\bbE_{\bsDelta} \left(\left\|\int_{[-1/2,1/2]^s} \scrK_{\bssigma,s}(t)\rd\bssigma - 
 \frac{1}{N}\sum_{k=0}^{N-1} \scrK_{\{\bsx_k + \bsDelta\},s}(t)\right\|_{Z}^2 \right)\le 
 C_{s,1,\bsgamma,\lambda}\, \frac{1}{(\phi_{\rm tot}(N))^{1/\lambda}},
\end{equation}
for all $\lambda\in (1/2,1]$, where $C_{s,1,\bsgamma,\lambda}$ is a constant that is independent of $N$, but dependent on the coordinate weights $\bsgamma=(\gamma_{\fraku})_{\fraku\subseteq \{1,\ldots,s\}}$ in the space $\calW_{s,1,\bsgamma,2,2}$,
\[
  C_{s,1,\bsgamma,\lambda}= \widetilde{C}^2 
  \left(\sum_{\emptyset \neq \fraku \subseteq \{1:s\}}
  \gamma_{\fraku}^{2\lambda} 
  \left(\frac{2\zeta(2\lambda)}{(2\pi^2)^\lambda}\right)^{|\fraku|}\right)^{\frac{1}{\lambda}}
  \left(\sum_{\fraku \subseteq \{1:s\}} 
  \frac{[(|\fraku | +2)!]^2\, \prod_{j\in \fraku} b_j^{2}}{\gamma_{\fraku}^2}\right).
\]

\subsection{Folded lattice rules for smoothness one}

As an alternative to a randomly shifted lattice rule, it is also possible to derive a deterministic error bound similar to \eqref{eq:error_1_random} by employing so-called \textit{folded lattice rules}, also referred to as \textit{tent-transformed lattice rules}. The tent transformation $\phi:[0,1]\to [0,1]$ is a Lebesgue measure preserving mapping defined by $\phi (x)=1-|2x-1|$. If we apply $\phi$ to a point $\bsx\in [0,1]^s$, we always mean component-wise application, i.e., $\phi (\bsx)=(\phi(x_1),\ldots,\phi(x_s))$. For a given lattice point set we then apply $\phi$ to all points, and thereby obtain the \textit{folded lattice point set}; the corresponding lattice 
rule is then called a folded lattice rule. 

Let us consider the same function space $\calV_{s,1,\bsgamma}$ as in Section \ref{sec:random_shift}. It is known from results in \cite{dick2014nonperiodic} that one can construct a folded lattice rule $Q_{N,s}^\phi$ by a CBC algorithm such that 
\[
 \left[e^{\rm wor} (Q_{N,s}^\phi,\calV_{s,1,\bsgamma})\right]^2 
 \le 
 \frac{1}{(\phi_{\rm tot}(N))^{1/\lambda}}\, 
 \left(\sum_{\emptyset \neq \fraku \subseteq \{1:s\}}
  \gamma_{\fraku}^{2\lambda} 
  \left(\frac{2\zeta(2\lambda)}{(\pi^2)^\lambda}\right)^{|\fraku|}\right)^{\frac{1}{\lambda}}
\]
for all $\lambda\in (1/2,1]$.

We can then proceed similarly to Section \ref{sec:random_shift}, and 
employ Theorem \ref{thm:GKKSS_Banach}. In this way, we can construct by a CBC algorithm a folded lattice rule with points $\phi (\bsx_0),\ldots,\phi(\bsx_{N-1})$ such that
\begin{equation}\label{eq:error_1_folded}
\left\|\int_{[-1/2,1/2]^s} \scrK_{\bssigma,s}(t)\rd\bssigma - 
 \frac{1}{N}\sum_{k=0}^{N-1} \scrK_{\{\phi(\bsx_k)\},s}(t)\right\|_{Z}^2 \le 
 \widehat{C}_{s,1,\bsgamma,\lambda}\, \frac{1}{(\phi_{\rm tot}(N))^{1/\lambda}},
\end{equation}
for all $\lambda\in (1/2,1]$, and where $\widehat{C}_{s,1,\bsgamma,\lambda}$ is a constant that is independent of $N$, but dependent on the coordinate weights $\bsgamma=(\gamma_{\fraku})_{\fraku\subseteq \{1,\ldots,s\}}$ in the space $\calW_{s,1,\bsgamma,2,2}$,
\[
 \widehat{C}_{s,1,\bsgamma,\lambda}= \widetilde{C}^2 
  \left(\sum_{\emptyset \neq \fraku \subseteq \{1:s\}}
  \gamma_{\fraku}^{2\lambda} 
  \left(\frac{2\zeta(2\lambda)}{(\pi^2)^\lambda}\right)^{|\fraku|}\right)^{\frac{1}{\lambda}}
  \left(\sum_{\fraku \subseteq \{1:s\}} 
  \frac{[(|\fraku | +2)!]^2\, \prod_{j\in \fraku} b_j^{2}}{\gamma_{\fraku}^2}\right),
\]
which means that the error bound using a folded lattice rule is, up to a factor of 2, the same as the probabilistic error bound using a randomly shifted lattice rule.

\subsection{Higher-order QMC}

We now consider the Sobolev space $\calW_{s,\alpha,\bsgamma,q,\infty}$. For integrands in 
$\calW_{s,\alpha,\bsgamma,q,\infty}$ it is natural to hope that 
the smoothness $\alpha$ is reflected in the error convergence rate of a cubature rule, in the sense that higher smoothness yields higher-order convergence. Indeed, this can be achieved by using higher-order QMC rules which are based on polynomial lattice rules, another important class of QMC node sets, which was introduced by Niederreiter (see, e.g., \cite{nied1992siam}). A polynomial lattice rule can be constructed analogously to an ordinary lattice rule, by replacing integer arithmetic by polynomial arithmetic over finite fields. The generating vector of a polynomial lattice point set is then a vector of polynomials over a finite field, and can again be obtained by a CBC construction. It was shown by Dick that one can use so-called higher-order polynomial lattice rules (which are obtained from ordinary polynomial lattice rules as building blocks, by the method of digit interlacing) for numerical integration in Sobolev spaces of higher smoothness and obtain essentially optimal error convergence rates. We refer to \cite{dick2010nets} for details. 
While previous studies of higher-order QMC methods consider real-valued integrads or integrands taking values in a separable Hilbert space, a novelty of our work is the generalization to separable Banach spaces.

Let~$\beta$ be prime,~$m \in \bbN$, and~$N=\beta^m$. Then, for~$r=\infty$ and every~$1\le q\le \infty$, it is known (see,~\cite[Thm.~3.10]{DKLNS14}) that an interlaced polynomial lattice rule~$Q_{N,s}$ of order~$\alpha$ can be constructed using a CBC algorithm, such that
\begin{align}\label{eq:cbc_wc}
    e^{\rm wor}(Q_{N,s}, \calW_{s,\alpha,\bsgamma,q,\infty}) \le \bigg(\frac{2}{N-1} \sum_{\emptyset \neq \fraku \subseteq \{1:s\}} \gamma_\fraku^\lambda (\rho_{\alpha,\beta}(\lambda))^{|\fraku|} \bigg)^{\frac{1}{\lambda}}
\end{align}
for all~$\lambda\in(1/\alpha,1]$, where
\begin{align*}
    \rho_{\alpha,\beta}(\lambda) := \left(C_{\alpha,\beta}\,\beta^{\alpha(\alpha-1)/2}\right)^\lambda \left( \left(1 + \frac{\beta-1}{\beta^{\alpha\lambda}-\beta} \right)^\alpha -1 \right), 
\end{align*}
with~$C_{\alpha,\beta}$ depending only on~$\alpha$ and~$\beta$.

Combining \eqref{eq:cbc_wc} with \eqref{eq:sup_bound_infinite}, we obtain
\[
\left\|\int_{[-1/2,1/2]^s} \scrK_{\bssigma,s}(t)\rd\bssigma - 
 \frac{1}{N}\sum_{k=0}^{N-1} \scrK_{\bsx_k,s}(t)\right\|_{Z}
 \le C_{s,\alpha,\bsgamma,\lambda} \frac{1}{(N-1)^{1/\lambda}}
\]
for all $\lambda\in (1/\alpha,1]$, where $C_{s,\alpha,\bsgamma,\lambda}$ is a constant that is independent of $N$, but dependent on the coordinate weights $\bsgamma=(\gamma_{\fraku})_{\fraku\subseteq \{1,\ldots,s\}}$ in the space $\calW_{s,\alpha,\bsgamma,q,\infty}$,
\begin{equation}\label{eq:err_alpha_nets}
 C_{s,\alpha,\bsgamma,\lambda}= \widetilde{C}  \bigg( 2 \sum_{\emptyset \neq \fraku \subseteq \{1:s\}} \gamma_\fraku^\lambda (\rho_{\alpha,b}(\lambda))^{|\fraku|} \bigg)^{\frac{1}{\lambda}}
 \frac{1}{\gamma_{\fraku}}
    \sum_{\bsnu_\fraku \in \{1:\alpha\}^{|\fraku|}} (|\bsnu_\fraku|+2)! \prod_{j\in \fraku} (2^{\delta(\nu_j,\alpha)} b_j^{\nu_j}).
\end{equation}

\medskip

\subsection{Choosing the weights}

Now, we make a particular choice of the weights $\bsgamma=(\gamma_{\fraku})_{\fraku\subseteq \{1,\ldots,s\}}$. 
These are so-called ``SPOD weights'' (smoothness-driven product and order dependent weights),
\begin{align}\label{eq:spod_weights}
    \gamma_\fraku^\ast := \sum_{\bsnu_\fraku \in \{1:\alpha\}^{|\fraku|}} (|\bsnu_\fraku|+2)! \prod_{j\in \fraku} (2^{\delta(\nu_j,\alpha)} b_j^{\nu_j}).
\end{align}
In the special case $\alpha=1$, the SPOD weights in \eqref{eq:spod_weights} simplify to ``POD weights'' (product and order dependent weights),
\begin{align}\label{eq:pod_weights}
    \gamma_\fraku^\ast :=  (|\fraku|+2)! \prod_{j\in \fraku} (2 b_j^{\nu_j}).
\end{align}

Let us start by analyzing the bound \eqref{eq:error_1_random} for randomly shifted lattice rules. From \cite[Lemma 6.2]{kuo2012qmcfe} we 
know that $C_{s,1,\bsgamma,\lambda}$ is minimized by choosing the 
set of weights $\bsgamma^*=(\gamma_{\fraku}^*)_{\fraku\subseteq \{1,\ldots,s\}}$, with
\[
\gamma_\fraku^* =\left((|\fraku| +2)!
\prod_{j\in\fraku} \frac{b_j (2\pi^2)^{\lambda/2}}{\sqrt{2\zeta (2\lambda)}}\right)^{1/(1+\lambda)}.
\]
Using the same argumentation as in \cite[Proof of Theorem 6.7]{guth2022parabolic} we then obtain that $C_{s,1,\bsgamma^*,\lambda}$ is bounded independently of $s$, by choosing 
\[
\lambda=\begin{cases}
            1/(2-2\delta)\quad  \mbox{for arbitrary $\delta\in (0,1)$} & 
            \mbox{if $p\in (0,2/3]$,}\\
            p/(2-p) & \mbox{if $p\in (2/3,1]$.}
        \end{cases}
\]
Moreover, it follows---again by the same argument as in \cite{guth2022parabolic}---that the bound \eqref{eq:error_1_random} on the mean-square error is of order 
\[
\kappa (N)=\begin{cases}
            [\phi_{\rm tot}(N)]^{-2-2\delta} \quad  \mbox{for arbitrary $\delta\in (0,1)$} & 
            \mbox{if $p\in (0,2/3]$,}\\
            [\phi_{\rm tot}(N)]^{-(2/p-1)} & \mbox{if $p\in (2/3,1]$.}
        \end{cases}
\]
For the bound \eqref{eq:error_1_folded} on the squared integration error 
using folded lattice rules in the case of smoothness one, we obviously can choose $\bsgamma$ and $\lambda$ in exactly the same way as for \eqref{eq:error_1_random}.

\medskip

Next, we insert the SPOD weights as in \eqref{eq:spod_weights} into \eqref{eq:err_alpha_nets}, and we make use of the inequality $\left(\sum_k a_k\right)^\lambda \le \sum_k a_k^\lambda$, which holds for any $\lambda\in (0,1]$ and nonnegative $a_k$. Consequently, we can bound the 
term $C_{s,\alpha,\bsgamma,\lambda}$ in \eqref{eq:err_alpha_nets},
\begin{align}
    C_{s,\alpha,\bsgamma,\lambda} &\le \widetilde{C}\bigg( 2 \sum_{\emptyset \neq \fraku \subseteq \{1:s\}} \sum_{\bsnu_\fraku \in \{1:\alpha\}^{|\fraku|}} ((|\bsnu_\fraku|+2)!)^\lambda \prod_{j \in \fraku} (E 2^{\delta(\nu_j,\alpha)} b_j^{\nu_j})^\lambda \bigg)^{\frac{1}{\lambda}}\notag\\
    &= \widetilde{C}\bigg( 2 \sum_{\boldsymbol{0}\neq \bsnu \in \{0:\alpha\}^s} ((|\bsnu|+2)!)^\lambda \prod_{\substack{j=1\\ \nu_j >0}}^s (E  2^{\delta(\nu_j,\alpha)} b_j^{\nu_j})^\lambda\bigg)^{\frac{1}{\lambda}},\label{eq:worweights}
\end{align}
where~$E:= C_{\alpha,\beta}\, \beta^{\alpha(\alpha-1)/2} ((1+\frac{\beta-1}{\beta^{\alpha \lambda}-\beta})^\alpha-1)^{\frac{1}{\lambda}}$. Following the steps in~\cite[pp.~2694--2695]{DKLNS14}, one shows that the sum in~\eqref{eq:worweights} is finite independently of~$s$ if~$p\le \lambda <1$. 
In the case~$\lambda = 1$, we assume in addition that~$\sum_{j\ge1}b_j < (2\alpha \max{(E,1)})^{-1}$. Since~$\lambda$ needs to satisfy~$1/\alpha<\lambda \le 1$, we take~$\lambda = p$ and $\alpha = \lfloor 1/p \rfloor +1$. In this case, the integration error is thus of order~$\mathcal{O}(N^{-\alpha})$, where the implied constant is independent of the dimension~$s$.

\section{Approximation error of trajectories and controls}\label{sec:errorprop}

In this section we investigate how the approximation error of the feedback propagates to the trajectories and the controls. We consider the affine feedback
\begin{align}\label{eq:averfb}
    K(t,\cdot) := \int_\fkS K_{\bssigma}(t,\cdot) \mathrm d\bssigma := -\int_\fkS B^\ast \Pi_{\bssigma}(T-t)(\cdot) \mathrm d\bssigma - \int_\fkS B^\ast h_{\bssigma}(t) \mathrm d\bssigma
\end{align}
for~$t\in [0,T]$, which covers both the homogeneous (Section~\ref{sec:homogeneous}) and the nonhomogeneous (Section~\ref{sec:nonhomogeneous}) cases.
Let us denote by
$$\hat{K}(t,\cdot) := \sum_{k=1}^N \alpha_k K_{\bssigma_{k},s}(t,\cdot) := - \sum_{k=1}^N \alpha_k B^\ast {\Pi}_{\bssigma_{k},s}(T-t) (\cdot) - \sum_{k=1}^N \alpha_k B^\ast {h}_{\bssigma_{k},s}(t)$$ 
an approximation of the feedback~$K(t,\cdot)$ using an~$N$-point cubature rule with weights~$\alpha_k \in \mathbb{R}$, and with~$s$-dimensional integration nodes~$\bssigma_{0},\ldots,\bssigma_{N-1}$ (as, for instance, described in Section~\ref{sec:QMC}), leading to the closed-loop system
\begin{align}\label{eq:yhat}
    \dot{\hat{y}}_{\bssigma} = (\calA_{\bssigma} + B \hat{K}) \hat{y}_{\bssigma} + f,\qquad \hat{y}_{\bssigma}(0) = y_\circ. 
\end{align}
Denoting by~$y_{\bssigma}$ the closed-loop state associated to the feedback law~$K$, as well as~$\delta y_{\bssigma} := y_{\bssigma} - \hat{y}_{\bssigma}$,~$B^\ast\Pi := \int_\fkS B^\ast \Pi_{\bssigma}\,\mathrm d\bssigma$,~$B^\ast \hat{\Pi} := \sum_{k=1}^N \alpha_k B^\ast \Pi_{\bssigma_{k},s}$,~$B^\ast\delta \Pi :=  B^\ast\Pi - B^\ast\hat{\Pi}$,~$B^\ast h := \int_\fkS B^\ast h_{\bssigma}\,\mathrm d\bssigma$,~$B^\ast \hat{h} := \sum_{k=1}^N \alpha_k B^\ast h_{\bssigma_{k},s}$, and~$B^\ast \delta h = B^\ast h- B^\ast\hat{h} $, we have 
\begin{align}\label{eq:deltay}
    \dot{\delta y} = (\calA_{\bssigma} - B B^\ast\Pi) \delta y - B (B^\ast \delta \Pi \hat{y}+B^\ast \delta h), \quad \delta y(0) = 0,
\end{align}
and we obtain the following result.
\begin{theorem}\label{thm:convtrajcont}
    Let~\eqref{A1},~\eqref{A2},~\eqref{A3},~\eqref{eq:A4},~\eqref{eq:A5}, and~\eqref{eq:A6} hold. Then, it holds that
\begin{align}
    \|y_{\bssigma}(t)-\hat{y}_{\bssigma}(t)\|_H &\le C_y\max_{t \in [0,T]} \left( \|B^\ast \delta \Pi(T-t)\|_{\calL(H,U)} + \|B^\ast \delta h(t) \|_U\right) \quad \forall t\in [0,T],\notag \\
    \|u_{\bssigma}(t)-\hat{u}_{\bssigma}(t)\|_U &\le C_u\max_{t \in [0,T]} \left( \|B^\ast \delta \Pi(T-t)\|_{\calL(H,U)} + \|B^\ast \delta h(t) \|_U\right)\quad \forall t\in [0,T],\notag
\end{align}
where the constants~$C_u$ and~$C_y$ are independent of~$\bssigma \in \fkS$.
\end{theorem}
\begin{proof}
By~\eqref{eq:boundPi} (with~$\bsnu = \boldsymbol{0}$) we have that~$BB^\ast\Pi \in L^\infty((0,T);\mathcal{L}(H))$ and~$BB^\ast\hat\Pi \in L^\infty((0,T);\mathcal{L}(H))$. Hence, for~$f \in V'_T$ and~$y_\circ \in H$, we have that~$y \in W_T(V,V')$ and~$\hat{y} \in W_T(V,V')$ (see, e.g., \cite[Thm.~1.2, Part II, Ch. 2]{bensoussan2007representation}), and thus~$\delta y \in W_T(V,V')$.
Testing~\eqref{eq:deltay} with~$\delta y$ leads to
\begin{align}
    &\frac{1}{2} \frac{\mathrm d}{\mathrm dt}\|\delta y_{\bssigma}(t)\|_H^2 = \langle \calA_{\bssigma} \delta y_{\bssigma}(t), \delta y_{\bssigma}(t) \rangle_{V',V} - \langle B B^\ast \Pi(T-t) \delta y_{\bssigma}(t), \delta y_{\bssigma}(t)\rangle_H \notag\\
        &\qquad\qquad\qquad\quad- \langle B( B^\ast \delta \Pi(T-t) \hat{y}_{\bssigma}(t) + B^\ast\delta h(t)), \delta y_{\bssigma}(t)\rangle_H\notag\\
    &\quad\le \rho \|\delta y_{\bssigma}(t)\|^2_H - \theta \|\delta y_{\bssigma}(t)\|^2_V + \|BB^\ast \Pi(T-t)\|_{\calL(H)} \|\delta y_{\bssigma}(t)\|^2_H \notag\\
        &\qquad\qquad\qquad\quad+ \frac{1}{2} \|B(B^\ast \delta \Pi(T-t) \hat{y}_{\bssigma}(t) + B^\ast \delta h(t)) \|_H^2 + \frac{1}{2}\|\delta y_{\bssigma}(t)\|^2_H\notag\\
    &\quad\le \bigg(\rho + \|BB^\ast \Pi(T-t)\|_{\calL(H)} + \frac{1}{2} \bigg) \|\delta y_{\bssigma}(t)\|^2_H + \frac{1}{2} \|B(B^\ast \delta \Pi(T-t) \hat{y}_{\bssigma}(t) + B^\ast \delta h(t) )\|_H^2, \notag
\end{align}
where we used~\eqref{A1},~\eqref{A2} with~$\theta \ge 0$, and Young's inequality. Since~$\sup_{t\in [0,T]}\|BB^\ast \Pi(T-t)\|_{\calL(H)} \le C_{B\Pi}$, Gronwall's lemma implies
\begin{align}
    \|\delta y_{\bssigma}(t) \|_H^2 &\le \int_0^t e^{(1+2(\rho+C_{B\Pi}))(t-s)}  \|B(B^\ast \delta\Pi(T-s) \hat{y}_{\bssigma}(s) + B^\ast\delta h(t)) \|_H^2\, \mathrm ds\notag\\
    &\le \max{\left(1,e^{(1+2(\rho+C_{B\Pi}))T}\right)} \|B\|^2_{\calL(H,U)} \left(\|B^\ast \delta \Pi \hat{y}_{\bssigma} + B^\ast \delta h\|_{U_T}^2 \right) \notag.
\end{align}
Thus, for all~$t\in [0,T]$, we have
\begin{align}\label{eq:deltaybound}
    \|\delta y_{\bssigma}(t) \|_H &\le \tilde C_{y} \bigg(\max_{s \in [0,T]} \left(\|B^\ast \delta \Pi(T-s)\|_{\calL(H,U)} +\|B^\ast \delta h(s)\|_U\right)\bigg),
\end{align}
for some constant~$\tilde C_{y} >0$ depending on~$B,\Pi,T,\hat{y}_{\bssigma}$.
Furthermore, testing~\eqref{eq:yhat} with~$\hat{y}$ leads after similar estimations to
\begin{align}\label{eq:yhelpdiff1}
    \|\hat y_{\bssigma} \|_{H_T}^2 \le T\max{\left(1,e^{(1+2(\rho+C_{B\hat\Pi}))T}\right)} \left(\|y_\circ\|_H^2 + \|BB^\ast \hat{h}_{\bssigma} + f\|_{H_T}^2\right),
\end{align}
where~$\sup_{t\in[0,T]}\|BB^\ast \hat\Pi(T-t)\|_{\calL(H)} \le C_{B\hat\Pi}$. 
Analogously, one can show that
\begin{align}\label{eq:yhelpdiff}
\|y_{\bssigma}(t)\|_H^2 \le \max{\left(1,e^{(1+2(\rho+C_{B\Pi}))T}\right)} \left(\|y_\circ\|_H^2 + \|BB^\ast h_{\bssigma} + f\|_{H_T}^2\right),\quad \forall t\in [0,T].
\end{align}
One may further apply Proposition~\ref{prop:hbound} to obtain bounds for~\eqref{eq:yhelpdiff1} and~\eqref{eq:yhelpdiff} which are independent of~$\bssigma \in \fkS$. Thus, the constant~$\tilde C_y$ in~\eqref{eq:deltaybound} can be bounded independently of~$\bssigma \in \fkS$, which proves the first result.

Furthermore, for the controls we obtain
\begin{align}
    u_{\bssigma}(t) - \hat{u}_{\bssigma}(t) &= - B^\ast\Pi(T-t)y_{\bssigma}(t) + B^\ast \hat{\Pi}(T-t) \hat{y}_{\bssigma}(t) - B^\ast\delta h(t)\notag \\ &= -B^\ast \delta \Pi(T-t) y_{\bssigma}(t) - B^\ast \hat{\Pi}(T-t) \delta y_{\bssigma}(t) -B^\ast \delta h(t)\notag
\end{align}
and thus
\begin{align}
    \|u_{\bssigma}(t)-\hat{u}_{\bssigma}(t)\|_U &\le \|B^\ast\delta \Pi(T-t)\|_{\calL(H,U)} \|y_{\bssigma}(t)\|_H\notag\\&\quad + \|B^\ast\hat{\Pi}(T-t)\|_{\calL(H,U)} \|\delta y_{\bssigma}(t)\|_H + \|B^\ast\delta h(t)\|_U \notag\\
    &\le C_u \bigg(\max_{s \in [0,T]} \left(\|B^\ast \delta \Pi(T-s)\|_{\calL(H,U)} +\|B^\ast \delta h(s)\|_U\right)\bigg),\notag
\end{align}
where we used the bound on the difference of the trajectories,~\eqref{eq:yhelpdiff}, as well as the fact that $\|B^\ast\hat{\Pi}\|_{\calL(H,U)}$ is bounded independently of~$\bssigma \in \fkS$ to make~$C_u$ independent of~$\bssigma \in \fkS$.
\end{proof}

We point out that Theorem~\ref{thm:convtrajcont} holds for arbitrary cubature rules. The regularity bounds obtained in Sect.~\ref{sec:analyticFB} can also be used for worst case error analysis of sparse grid integration, for instance.

\subsection{Suboptimality of the feedback}
Suppose there is a `true' parameter~$\bar\bssigma \in \fkS$. Let us denote by~$K_{\bar\bssigma}(t,\cdot) = - B^\ast \Pi_{\bar\bssigma}(T-t)(\cdot) - B^\ast h_{\bar\bssigma}(t)$ the optimal feedback for~${\bar\bssigma} \in \fkS$, leading to the closed-loop system
\begin{align}
    \dot{y}_{\bar\bssigma} = (\calA_{\bar\bssigma} + B {K}_{\bar\bssigma}) {y}_{\bar\bssigma} + f,\qquad {y}_{\bar\bssigma}(0) = y_\circ. \notag
\end{align}

The trajectory and control corresponding to the feedback~\eqref{eq:averfb} are close to the optimal trajectory and optimal control for~$\bar \bssigma$ provided that the feedback~$K$ is close to~$K_{\bar\bssigma}$ meaning that~$B^\ast \Pi := \int_\fkS B^\ast \Pi_{\bssigma} \mathrm d\bssigma$ and~$B^\ast h := \int_\fkS B^\ast h_{\bssigma} \mathrm d\bssigma$ are close to~$B^\ast \Pi_{\bar\bssigma}$ and~$B^\ast h_{\bar\bssigma}$, respectively.
Writing~$\delta y_{\bar\bssigma} := {y}_{\bar\bssigma} - y_{\bssigma} $,~$B^\ast\delta \Pi_{\bar\bssigma} :=  B^\ast\Pi_{\bar\bssigma} - B^\ast{\Pi}$, and~$B^\ast \delta h_{\bar\bssigma} = B^\ast h_{\bar\bssigma}- B^\ast {h}$, we find
\begin{align}
    \dot{\delta y}_{\bar\bssigma} = (\calA_{\bar\bssigma} - B B^\ast\Pi_{\bar\bssigma}) \delta y_{\bar\bssigma} - B B^\ast (\delta \Pi_{\bar\bssigma} {y}+\delta h_{\bar\bssigma}), \quad \delta y_{\bar\bssigma}(0) = 0, \notag
\end{align}
and the subsequent result follows.
\begin{theorem}\label{thm:suboptimal}
    Let~\eqref{A1},~\eqref{A2},~\eqref{A3},~\eqref{eq:A4},~\eqref{eq:A5}, and~\eqref{eq:A6} hold. Then, it holds that
\begin{align}
    \|y_{\bar\bssigma}(t)-{y}_{\bssigma}(t)\|_H &\le \bar C_y\max_{t \in [0,T]} \left( \|B^\ast \delta \Pi_{\bar\bssigma}(T-t)\|_{\calL(H,U)} + \|B^\ast \delta h_{\bar\bssigma}(t) \|_U\right) \quad \forall t\in [0,T],\notag \\
    \|u_{\bar\bssigma}(t)-{u}_{\bssigma}(t)\|_U &\le \bar C_u\max_{t \in [0,T]} \left( \|B^\ast \delta \Pi_{\bar\bssigma}(T-t)\|_{\calL(H,U)} + \|B^\ast \delta h_{\bar\bssigma}(t) \|_U\right)\quad \forall t\in [0,T],\notag
\end{align}
where the constants~$\bar C_u$ and~$\bar C_y$ are independent of~$\bssigma \in \fkS$.
\end{theorem}
The proof of Theorem~\ref{thm:suboptimal} follows the same steps as the proof of Theorem~\ref{thm:convtrajcont}. For this reason, Theorem~\ref{thm:suboptimal} also remains true when~$K$ is replaced by its approximation~$\hat{K}$.



\appendix
\section{Proof of Lemma~\ref{lem:apriori_unc}}\label{sec:appendix}

\begin{proof}
Let~$f\in V'_T$ and~$y_\circ \in H$ be arbitrary. We will show that the solution of
\begin{align}\label{eq:eqn}
    \dot{y}_{\bssigma} = \calA_{\bssigma} y_{\bssigma} + f,\qquad y_{\bssigma}(0) = y_\circ,
\end{align}
satisfies the~\emph{a priori} estimate
\begin{align}
    \|y_{\bssigma}\|_{W_T(V,V')} \le c_A(T) \big(\|f\|_{V'_T}^2 + \|y_\circ\|_{H}^2\big)^{\frac{1}{2}}, \quad \forall \bssigma \in \fkS\notag
\end{align}
with~$T\mapsto c_A(T)$ continuous and monotonically increasing and independent of~$\bssigma \in \fkS$, which proves the result. To this end, we test~\eqref{eq:eqn} against~$y$, leading, together with~\eqref{A2}, to
    \begin{align}
        \frac{\rm d}{\rm dt} \|y\|_H^2 &\le -2 \theta \|y\|_V^2 + 2 \rho \|y\|_H^2 + 2\|f\|_{V'} \|y\|_V\label{eq:ex1help}\\ &\le -2 \theta \|y\|_V^2 + 2 \rho \|y\|_H^2 + \frac{1}{2\theta} \|f\|_{V'}^2 + 2\theta \|y\|_V^2 \le 2 \rho \|y\|_H^2 +\frac{1}{2\theta} \|f\|_{V'}^2,\notag
    \end{align}
    which implies with Gronwall's lemma that 
    \begin{align}\label{eq:ex1help2}
        \|y\|^2_H \le e^{2\rho T} \left( \|y_\circ\|^2_H + \|f\|_{V'_T}^2 \right).
    \end{align}
    Furthermore, using Young's inequality, we obtain from~\eqref{eq:ex1help} that
    \begin{align}
        \frac{\rm d}{\rm dt} \|y\|_H^2  + \theta \|y\|_V^2 &\le 2 \rho \|y\|_H^2 + \frac{1}{\theta} \|f\|_{V'}^2,\notag
    \end{align}
    which, after integration over~$[0,T]$, gives
    \begin{align}
        \|y(T)\|_H^2 + \theta \int_0^T \|y\|_V^2\mathrm dt   &\le \|y(0)\|_H^2 + 2 \rho \int_0^T \|y\|_H^2 \mathrm dt + \frac{1}{\theta} \|f\|_{V'_T}^2,\label{eq:yT-bound}
    \end{align}
    and thus 
    \begin{align}\label{eq:ex1help3}
        \|y\|_{L^2(0,T;V)}^2 \le \frac{1}{\theta}\|y_\circ\|^2_H + \frac{2}{\theta}\rho T e^{2\rho T} \left( \|y_\circ\|^2_H + \|f\|_{V'_T}^2 \right)+ \frac{1}{\theta^2} \|f\|_{V'_T}^2,
    \end{align}
    where we used~\eqref{eq:ex1help2}.
    Moreover, from~\eqref{eq:eqn} we obtain that
    \begin{align}\label{eq:ex1help4}
        \|\dot{y}\|^2_{V'} &\le 2\|\calA y\|_{V'}^2 + 2\|f\|_{V'}^2 \le 2C_\calA^2 \|y\|_V^2 + 2 \|f\|_{V'}^2.
    \end{align}
    Finally, integration of~\eqref{eq:ex1help4} together with~\eqref{eq:ex1help3} gives
    \begin{align}
        \|y\|_{W_T(V,V')}^2 
        &=  \|y(t)\|_{V_T}^2 +  \|\dot{y}(t)\|^2_{V'_T} \le c_A(T)^2 \left( \|y_\circ\|^2_H + \|f\|_{V'_T}^2\right), \label{eq:uncapriori}
    \end{align}
    where~$c_A(T) := \left((1+2C_\calA)  \frac{1 + 2\rho T e^{2\rho T}}{\theta} + (1+2C_\calA) \frac{2}{\theta} \rho T e^{2\rho T} + \frac{(1+2C_\calA)}{\theta^2} +  2\right)^{\frac{1}{2}}$, which is continuous and monotonically increasing in~$T$.
\end{proof}

\section{An auxiliary result}
The following result is motivated by~\cite{guka_prep}.
\begin{lemma}\label{lem:recursion}
Let $(\Upsilon_{\bsnu})_{\bsnu\in\mathcal F}$ and~$\bsb = (b_j)_{j\ge 1}$ be sequences of nonnegative numbers satisfying
$$
\Upsilon_{\bsnu}\leq \frac{1}{2}C\sum_{\substack{\boldsymbol m\leq \bsnu\\ \boldsymbol m\neq\mathbf 0}}\binom{\bsnu}{\boldsymbol m}(|\boldsymbol m|+2)! \boldsymbol b^{\boldsymbol m}\Upsilon_{\bsnu-\boldsymbol m}+\frac{1}{2}C(|\bsnu|+2)! \boldsymbol b^{\bsnu}\quad\text{for all}~\bsnu\in\mathcal F,
$$
where~$C>0$ is a constant. Then, it holds that
$$
\Upsilon_{\bsnu}\leq \frac{1}{2}(1+C)^{\max\{|\bsnu|-1,0\}}C^{\delta_{\bsnu,\mathbf 0}}(C+C^2)^{1-\delta_{\bsnu,\mathbf 0}}(|\bsnu|+2)! \boldsymbol b^{\bsnu}.
$$
\end{lemma}
\begin{proof}
We prove the result by induction on~$|\bsnu|$. For~$\bsnu = \boldsymbol{0}$ we have~$\Upsilon_{\boldsymbol{0}}\le C$. Let $\bsnu\in\mathcal F\setminus\{\mathbf 0\}$. Below, we write $\boldsymbol m\leq \bsnu$ for two multi-indices $\boldsymbol m$ and $\bsnu$ if the inequality holds component-wise. Assuming that the claim has been proved for all multi-indices with order less than $|\bsnu|$, we arrive at
\begin{align*}
&\Upsilon_{\bsnu}\leq \frac{1}{2}C\sum_{\substack{\boldsymbol m\leq \bsnu\\ \boldsymbol m\neq\mathbf 0}}\binom{\bsnu}{\boldsymbol m} \bigg((|\boldsymbol m|+2)! \boldsymbol b^{\boldsymbol m} \frac{1}{2}(1+C)^{\max\{|\bsnu|-|\boldsymbol m|-1,0\}}C^{\delta_{\bsnu-\boldsymbol m,\mathbf 0}}(C+C^2)^{1-\delta_{\bsnu-\boldsymbol m,\mathbf 0}}\notag \\&\quad\quad \times ((|\bsnu|-|\boldsymbol m|+2)!) \boldsymbol b^{\bsnu-\boldsymbol m}\bigg) +\frac{1}{2}C(|\bsnu|+2)! \boldsymbol b^{\bsnu}.
\end{align*}
Separating out the~$\bsm = \bsnu$ term and denoting~$\boldsymbol{\Lambda}_{C,|\bsnu|,\bsb} := \frac{1}{2}C(1+C)(|\bsnu|+2)! \boldsymbol b^{\bsnu}$ leads to
\begin{align*}
&\Upsilon_{\bsnu}\leq \frac{1}{4} C (C+C^2)\!\!\!\sum_{\substack{\boldsymbol m\leq \bsnu\\ \mathbf 0 \neq \boldsymbol m \neq \bsnu}}\!\!\!\binom{\bsnu}{\boldsymbol m} (|\boldsymbol m|+2)! \boldsymbol b^{\boldsymbol m}(1+C)^{|\bsnu|-|\boldsymbol m|-1} ((|\bsnu|-|\boldsymbol m|+2)!) \boldsymbol b^{\bsnu-\boldsymbol m} + \boldsymbol{\Lambda}_{C,|\bsnu|,\bsb}\\
&= \frac{1}{4}C(C+C^2)\boldsymbol b^{\bsnu}\sum_{\ell=1}^{|\bsnu|-1}(|\bsnu|-\ell+2)!(\ell+2)! (1+C)^{|\bsnu|-\ell-1}\frac{|\bsnu|!}{(|\bsnu|-\ell)!\ell!} + \boldsymbol{\Lambda}_{C,|\bsnu|,\bsb}\\
&=\frac{1}{4} C(C+C^2)\boldsymbol b^{\bsnu}\sum_{\ell=1}^{|\bsnu|-1} |\bsnu|! (|\bsnu|-\ell+2)(|\bsnu|-\ell+1)(\ell+2)(\ell+1) (1+C)^{|\bsnu|-\ell-1}+ \boldsymbol{\Lambda}_{C,|\bsnu|,\bsb}\\
&= \frac{1}{4} C(C+C^2)\boldsymbol b^{\bsnu} |\bsnu|!\sum_{\ell=1}^{|\bsnu|-1}  2(\ell+2)(\ell+1) (1+C)^{|\bsnu|-\ell-1} + \boldsymbol{\Lambda}_{C,|\bsnu|,\bsb}\\
&\le \frac{1}{2}C(C+C^2)\boldsymbol b^{\bsnu} (|\bsnu|+2)!\sum_{\ell=1}^{|\bsnu|-1} (1+C)^{|\bsnu|-\ell-1} +\boldsymbol{\Lambda}_{C,|\bsnu|,\bsb}\\
&= (|\bsnu|+2)! \bsb^\bsnu \bigg(\frac{1}{2}C(1+C) + \frac{1}{2}C(C+C^2) \sum_{\ell=1}^{|\bsnu|-1} (1+C)^{|\bsnu|-\ell-1}\bigg)\\
&= (|\bsnu|+2)! \bsb^\bsnu \bigg(\frac{1}{2}C(1+C) + \frac{1}{2}C(C+C^2) \frac{(C+1)^{|\bsnu|}-C-1}{C+C^2}\bigg)\\&= (|\bsnu|+2)! \bsb^\bsnu \frac{1}{2}C (C+1)^{|\bsnu|}.
\end{align*}
This implies the claim.
\end{proof}

\section{Consistency of the dimension truncation}
\label{sec:appendixdimtrunc}
Let~$\calA_{\bssigma,s} := \calA((\bssigma_s,\boldsymbol{0})) = \calA((\sigma_1,\sigma_2,\ldots,\sigma_s,0,0,\ldots ))$ and suppose that~\eqref{eq:A4} holds true. In this section we show that
\begin{align}\label{eq:Aconv}
    \calA_{\bssigma,s}x \overset{s \to \infty}{\longrightarrow} \calA_{\bssigma} x\quad \forall x \in D(\calA)\qquad \text{and} \qquad \calA_{\bssigma,s}^\ast y \overset{s \to \infty}{\longrightarrow} \calA_{\bssigma}^\ast y \quad \forall y \in D(\calA^\ast)
\end{align} 
imply that~$\|\scrK_{\bssigma}(t) - \scrK_{\bssigma,s}(t)\|_Z \overset{s \to \infty}{\longrightarrow} 0$ for~$\scrK_{\bssigma} \in \{-B^\ast \Pi_{\bssigma}, -B^\ast h_{\bssigma}\}$, and their~$s$-dimensional counterparts, with~$Z \in \{\calL(H,U),U\}$.

We begin with the homogeneous case~$\scrK_{\bssigma} = -B^\ast \Pi_{\bssigma}$. To do so, let~$(y_{\bssigma,s},u_{\bssigma,s})$ denote the minimizer of the dimensionally truncated problem of minimizing~\eqref{eq:J1} subject to~\eqref{eq:psys} with~$\calA_{\bssigma}$ replaced by~$\calA_{\bssigma,s}$. 

The operator~$\delta \Pi_{\bssigma} = \Pi_{\bssigma,s}-\Pi_{\bssigma}$ is bounded, linear, and self-adjoint. Thus, using~\cite[Thm.~5.23.8]{naylorsell}, we have
\begin{align}
    &\|\Pi_{\bssigma,s}-\Pi_{\bssigma}\|_{\calL(H)}  = \sup_{\|y_\circ\|_H = 1} | \calJ(y_{\bssigma,s},u_{\bssigma,s}) - \calJ(y_{\bssigma},u_{\bssigma})|\label{eq:PIsconv} \\
    &\le \sup_{\|y_\circ\|_H = 1} \frac{1}{2} \left(\|Q\|^2_{\calL(H)} \|y_{\bssigma,s} - y_{\bssigma}\|_{H_T} \|y_{\bssigma,s} + y_{\bssigma}\|_{H_T}+\|u_{\bssigma,s} - u_{\bssigma}\|_{U_T} \|u_{\bssigma,s} + u_{\bssigma}\|_{U_T}\right. \notag\\ &\quad\quad+\left.\|P\|^2_{\calL(H)} \|y_{\bssigma,s}(T) - y_{\bssigma}(T)\|_{H} \|y_{\bssigma,s}(T) + y_{\bssigma}(T)\|_{H}\right).\notag
\end{align}
Hence, it remains to show that~$y_{\bssigma,s} \to y_{\bssigma}$ in~$C([0,T];H)$ and~$u_{\bssigma,s} \to u_{\bssigma}$ in~$U_T$ as~$s\to\infty$.

As a consequence of the cost functional~$y_{\bssigma,s}$ and~$u_{\bssigma,s}$ are bounded uniformly in~$s$ in~$H_T$ and~$U_T$, respectively. Hence, there exist~$\widetilde{y}_{\bssigma,s}$ and~$\widetilde{u}_{\bssigma,s}$ such that~$y_{\bssigma,s} \rightharpoonup \widetilde{y}_{\bssigma,s}$ in~$H_T$ and~$u_{\bssigma,s} \rightharpoonup \widetilde{u}_{\bssigma,s}$ in~$U_T$.

Next, let~$\bar{u} \in U_T$ be arbitrary and let~$\bar{y}_{\bssigma,s}$ and~$\bar{y}_{\bssigma}$ be the solutions of
\begin{align}\label{eq:sconvy2}
    \dot{\bar{y}}_{\bssigma} = \calA_{\bssigma} \bar{y}_{\bssigma} + B \bar{u},\qquad \text{and} \qquad
    \dot{\bar{y}}_{\bssigma,s}= \calA_{\bssigma,s} \bar{y}_{\bssigma,s} + B \bar{u},
\end{align}
under the same initial condition~$\bar{y}_{\bssigma}(0) = \bar{y}_{\bssigma,s} = y_\circ$. Setting~$e_{\bssigma,s} := \bar{y}_{\bssigma} - \bar{y}_{\bssigma,s}$, we find
\begin{align}\label{eq:Leb}
e_{\bssigma,s}(t) = \left(S_{\bssigma}(t)-S_{\bssigma,s}(t)\right)y_\circ + \int_0^t \left(S_{\bssigma}(t-\tau)-S_{\bssigma,s}(t-\tau)\right) B \bar{u}(\tau)\,\mathrm d\tau
\end{align}
where~$S_{\bssigma}$ and~$S_{\bssigma,s}$ are the semigroups generated by~$\calA_{\bssigma}$ and~$\calA_{\bssigma,s}$, respectively. With the fact that $\|S_{\bssigma}(t)\|_{\calL(H)} \le M e^{-\rho t}$ and~$\|S_{\bssigma,s}(t)\|_{\calL(H)} \le M e^{-\rho t}$, as well as the assumption that~$\calA_{\bssigma,s} x \to  \calA_{\bssigma} x$ for all~$x \in D(\calA)$, the Trotter--Kato theorem (see~\cite[Thm.~4.5]{pazy}) implies~$S_{\bssigma,s}(t)x \to S_{\bssigma}(t)x$ for all~$x \in H$ uniformly for all~$t \in [0,T]$. Moreover, the fact that~$\|S_{\bssigma}(t)\|_{\calL(H)} \le M e^{-\rho t}$ and~$\|S_{\bssigma,s}(t)\|_{\calL(H)} \le M e^{-\rho t}$, Lebesgue's dominated convergence theorem, and~\eqref{eq:Leb} imply that
\begin{align}\label{eq:sconvy}
    \sup_{t \in [0,T]} \|e_{\bssigma,s}(t)\|_H \to 0 \quad \text{ as } \quad s \to \infty.
\end{align}

Next, we show that~$y_{\bssigma,s} \rightharpoonup\widetilde{y}_{\bssigma} = \widetilde{y}_{\bssigma}(\widetilde{u}_{\bssigma})$. To this end, we observe for~$v \in D(\calA^\ast)$ that
\begin{align}\label{eq:weakconvhelp}
     \langle y_{\bssigma,s}, \calA_{\bssigma,s}^\ast v \rangle_{H_T} = \underbrace{\langle y_{\bssigma,s}, (\calA_{\bssigma,s}^\ast - \calA_{\bssigma}^\ast) v\rangle_{H_T}}_{\rm term_1} + \underbrace{\langle  y_{\bssigma,s} - \widetilde{y}_{\bssigma}, \calA_{\bssigma}^\ast v\rangle_{H_T}}_{\to 0 \text{ as } s\to \infty} + \langle \widetilde{y}_{\bssigma}, \calA_{\bssigma}^\ast v\rangle_{H_T},
\end{align}
where~${\rm term_1} \le \|(A_{\bssigma,s}^\ast - \calA_{\bssigma}^\ast)v\|_H \|y_{\bssigma,s}\|_{H_T} \to 0$ as~$s \to \infty$ by~\eqref{eq:Aconv}.

Further, integrating~$\dot{y}_{\bssigma,s}(\tau) = \calA_{\bssigma,s} y_{\bssigma,s}(\tau) + B u_{\bssigma,s}(\tau)$ with~$y_{\bssigma,s}(0) = y_\circ$ over~$[0,t]$ leads to
\begin{align*}
    \langle y_{\bssigma,s}(t) - y_\circ, v\rangle_H = \int_0^t \left(\langle y_{\bssigma,s}(\tau) , \calA_{\bssigma,s}^\ast v \rangle_H + \langle Bu_{\bssigma,s}(\tau), v\rangle_H \right) \mathrm d\tau \quad \forall v \in D(\calA^\ast).
\end{align*}
Combining the weak convergence of~$y_{\bssigma,s}$ and~$u_{\bssigma,s}$ with~\eqref{eq:weakconvhelp} we obtain
\begin{align*}
    \langle \widetilde y_{\bssigma}(\tau) - y_\circ, v\rangle_H = \int_0^T \left(\langle \widetilde y_{\bssigma}(\tau) , \calA_{\bssigma}^\ast v \rangle_H + \langle B\widetilde u_{\bssigma}(\tau), v\rangle_H \right) \mathrm d\tau \quad \forall v \in D(\calA^\ast),
\end{align*}
which shows that~$y_{\bssigma,s} \rightharpoonup \widetilde{y}_{\bssigma} = {y}_{\bssigma}(\widetilde{u}_{\bssigma})$.

In order to show optimality of~$\widetilde{u}_{\bssigma}$ let~$\hat{y}_{\bssigma,s}$ denote the solution of~$\dot{\hat{y}}_{\bssigma,s} = \calA_{\bssigma,s} \hat{y}_{\bssigma,s} + B u_\bssigma$ with $\hat{y}_{\bssigma,s}(0) = y_\circ$. By the weak lower semicontinuity of~$\calJ$ we obtain \begin{align}\label{eq:dimtrunC7}
    \calJ(\widetilde{y}_{\bssigma},\widetilde{u}_{\bssigma}) \le \liminf_{s\to \infty} \calJ(y_{\bssigma,s},u_{\bssigma,s}) \le \limsup_{s \to \infty} \calJ(y_{\bssigma,s},u_{\bssigma,s}) \le \limsup_{s \to \infty} \calJ(\hat{y}_{\bssigma,s},u_{\bssigma}) = \calJ(y_{\bssigma},u_{\bssigma}),
\end{align}
where the last step follows from~\eqref{eq:sconvy} by taking~$\bar{u} = u_\bssigma$ in~\eqref{eq:sconvy2}. By uniquenes of the solution to~\eqref{eq:J1} subject to~\eqref{eq:psys} we have that~$
(\widetilde{y}_{\bssigma},\widetilde{u}_\bssigma) = (y_\bssigma,u_\bssigma)$. 
It remains to show that~$u_{\bssigma,s}\to u_{\bssigma}$ strongly. 
For this purpose we use~\eqref{eq:dimtrunC7} and~\cite[Lem.~5.2]{casaskunisch23partII}, which implies that~$u_{\bssigma,s}$ converges in norm to~$u_{\bssigma}$. Together with the weak convergence this implies strong convergence of~$u_{\bssigma,s}$ to~$u_{\bssigma}$. With~\eqref{eq:PIsconv} we conclude that~$\|\Pi_{\bssigma,s} - \Pi_{\bssigma}\|_{\calL(H)} \to 0 $ as~$s\to \infty$, which proves the claim for~$\scrK_{\bssigma} = -B^\ast \Pi_{\bssigma}$.

Furthermore, we conclude that~$\left( \calA_{\bssigma,s}^\ast - BB^\ast\Pi_{\bssigma,s}\right) x \overset{s\to \infty}{\longrightarrow} \left(\calA_{\bssigma}^\ast - BB^\ast\Pi_{\bssigma} \right) x$ for all~$x\in D(\calA^\ast)$, and thus by the Trotter--Kato theorem we have for the semigroups~$\widetilde{S}_{\bssigma}$ and~$\widetilde{S}_{\bssigma,s}$, generated by~$\calA_{\bssigma}^\ast - BB^\ast\Pi_{\bssigma}$ and~$\calA_{\bssigma,s}^\ast - BB^\ast\Pi_{\bssigma,s}$, that~$\widetilde{S}_{\bssigma}x \overset{s\to \infty}{\longrightarrow} \widetilde{S}_{\bssigma,s} x $ for all~$x \in H$ uniformly for all $t \in[0,T]$.

Let~$h_{\bssigma}$ be the solution of~\eqref{eq:h}, and let~$\hat{h}_{\bssigma,s}$, and $h_{\bssigma,s}$ be the solutions of
\begin{align*}
    -\dot{\hat{h}}_{\bssigma,s}(t) &= \left(\calA_{\bssigma,s}^\ast - \Pi_{\bssigma,s}(T-t)BB^\ast\right) \hat{h}_{\bssigma,s}(t) + \Pi_{\bssigma}(T-t) \left( f(t) + \calA_{\bssigma} g(t) - \dot{g}(t)\right),\\
    -\dot{h}_{\bssigma,s}(t) &= \left(\calA_{\bssigma,s}^\ast - \Pi_{\bssigma,s}(T-t)BB^\ast\right) h_{\bssigma,s}(t) + \Pi_{\bssigma,s}(T-t) \left( f(t) + \calA_{\bssigma,s} g(t) - \dot{g}(t)\right),
\end{align*}
which are bounded uniformly in~$s$ in~$C([0,T];H)$ by taking~$\bsnu = \boldsymbol{0}$ in Proposition~\ref{prop:hbound}.
Then, we have
\begin{align}\label{eq:suph}
    \sup_{t\in [0,T]}\|h_{\bssigma} - h_{\bssigma,s}\|_{H} &\le \sup_{t\in [0,T]} \left(\|h_{\bssigma} - \hat{h}_{\bssigma,s}\|_{H} +\|\hat{h}_{\bssigma,s} - h_{\bssigma,s}\|_{H} \right).
\end{align}

For the first term on the right-hand side of \eqref{eq:suph} we have
\begin{align*}
    h_{\bssigma}(t) - \hat{h}_{\bssigma,s}(t) = \int_0^t \left( \widetilde{S}_{\bssigma}(t-\tau) - \widetilde{S}_{\bssigma,s}(t-\tau)\right) \left( \Pi_{\bssigma}(t-\tau) \left( f(\tau) + \calA_{\bssigma} g(\tau) - \dot{g}(\tau)\right)\right) \mathrm d\tau.
\end{align*}
By the Trotter--Kato theorem and Lebesgue's dominated convergence theorem it follows that $\sup_{t\in [0,T]}\|h_{\bssigma}- \hat{h}_{\bssigma,s}\|_{H} \to 0$ as~$s\to \infty$.

For the second term we have
\begin{align*}
    \hat{h}_{\bssigma,s}(t) - h_{\bssigma,s}(t) = \int_0^t  \widetilde{S}_{\bssigma,s}(t-\tau) \left( \Pi_{\bssigma,s}(t-\tau)  \calA_{\bssigma,s} g(t)  - \Pi_{\bssigma}(t-\tau) \calA_{\bssigma} g(t)\right) \mathrm d\tau,
\end{align*}
and hence by Lebesgue's dominated convergence theorem we have~$\sup_{t\in [0,T]}\|\hat{h}_{\bssigma,s}-h_{\bssigma,s} \|_{H} \to 0$ as~$s\to \infty$, which by~\eqref{eq:suph} proves the claim for~$\scrK = -B^\ast h_{\bssigma}$.

\begin{remark}\label{rem:affine_As_conv}
    Let~$\calA_{\bssigma} = A_0 + \sum_{j \ge 1} \sigma_j A_j$ depend in an affine way on the parameters (as in Section~\ref{sec:affine}). Then, we estimate
    \begin{align*}
        \|(\calA_{\bssigma} - \calA_{\bssigma,s}) v\|_{H} = \bigg\|\sum_{j \ge s+1} \sigma_j A_j v\bigg\|_{H} \le \frac{ \|v\|_{D(\calA)} }{2} \sum_{j \ge s+1} \|A_j\|_{\calL(D(\calA),H)}.
    \end{align*}
Taking the limit as~$s \to \infty$, we get~$\lim_{s \to \infty} \|(\calA_{\bssigma} - \calA_{\bssigma,s}) v\|_{H} \to 0$ for all~$v \in D(\calA)$. Analogously, the convergence for the adjoint operator is shown. Thus, a parameterized family of operators depending in an affine way on the parameters satisfies~\eqref{eq:Aconv} and consequently~\eqref{eq:limit_K_s}.
\end{remark}

\section*{Acknowledgments}

P.~Kritzer gratefully acknowledges the support of the Austrian Science Fund (FWF) Project P34808. For open access purposes, the authors have applied a CC BY public copyright license to any author accepted manuscript version arising from this submission.

\bibliographystyle{plain}
\bibliography{references}

\end{document}